\documentclass[a4paper]{article}

\usepackage[centertags]{amsmath}
\usepackage{amsthm}
\usepackage{fullpage}

\usepackage{amssymb,amsfonts}

\usepackage[colorinlistoftodos,prependcaption]{todonotes}

\usepackage{xcolor}
\usepackage{graphicx}
\definecolor{cb2red}{RGB}{228,26,28}
\definecolor{cb2blue}{RGB}{55,126,184}
\definecolor{cb2green}{RGB}{77,175,74}
\definecolor{cb2orange}{RGB}{255,127,0}

\RequirePackage{doi}
\usepackage{hyperref}
\hypersetup{colorlinks,
    linkcolor=cb2blue,
    filecolor=cb2blue,
    citecolor=cb2blue,
    urlcolor=cb2blue}
\usepackage{url}

\usepackage{enumitem}
\usepackage{subcaption}
\usepackage{booktabs}
\usepackage{mathtools}

\usepackage{tikz}
\usepackage{pgfplots}
\pgfplotsset{compat=1.18}
\usepgfplotslibrary{groupplots}

\newcommand{\R}{\mathbb{R}}
\newcommand{\N}{\mathbb{N}}

\newcommand{\coloneqq}{:=}
\DeclareMathOperator*{\argmin}{\arg\min}
\DeclareMathOperator*{\minimize}{minimize}

\DeclareMathOperator{\subjectto}{subject~to}
\DeclareMathOperator{\mand}{\mathrm{and}}

\DeclareMathOperator{\Fix}{Fix}
\DeclareMathOperator{\prox}{prox}

\DeclareMathOperator{\dom}{dom}
\DeclareMathOperator{\ran}{ran}

\renewcommand{\H}{\ensuremath{\mathcal{H}}}

\newcommand{\gr}[1]{\mathrm{gra}({#1})}

\newcommand{\conv}{\mathrm{conv}}

\newcommand{\elcap}{\ifmmode \ell_{1}\text{--cap}\else $\ell_{1}$--cap\fi}

\newtheorem{theorem}{Theorem}[section]
\newtheorem{lemma}[theorem]{Lemma}
\newtheorem{proposition}[theorem]{Proposition}
\newtheorem{corollary}[theorem]{Corollary}
\theoremstyle{definition}
\newtheorem{definition}[theorem]{Definition}
\newtheorem{remark}[theorem]{Remark}
\newtheorem{assumption}[theorem]{Assumption}

\title{\bfseries Algorithm duality from resolvent Moreau identities}
\author{Andrew Calcan\thanks{Centre for Optimisation and Decision Science, Curtin University, Australia}\and%
    Alberto De~Marchi\thanks{Department of Aerospace Engineering, University of the Bundeswehr Munich, Germany}\and%
    Scott B. Lindstrom$^*$}

\begin{document}
\maketitle

\newcommand{\mykeywords}{Duality, Gabay duality, Moreau's identity, resolvents, splitting methods, nonconvex optimization, Douglas--Rachford method, ADMM, LASSO, regression.} 
\newcommand{\mymsccodes}{49N15, 
65K10, 
90C46.
} 

\begin{abstract}
Duality is classically defined as a relationship between objective functions, one that breaks down beyond convexity. We show that a more fundamental duality persists between the resolvent operators used to solve associated problems. We establish a framework from which we obtain complete characterizations of the duality--like correspondences between proximal fixed--point iterations and their alternating minimization scheme counterparts. Namely, we show the known primal-dual relationships between forward--backward splitting and Tseng's alternating minimization algorithm, and between the Douglas--Rachford method and alternating direction method of multipliers, to hold generically and irrespective of problem assumptions. The approach considered herein avoids the machinery of convex analysis entirely, providing a definitive formulation of algorithmic dualities in maximal generality while greatly simplifying their proofs. In particular, assumptions such as monotonicity, weak monotonicity, Lipschitz continuity, and injectivity of linear operators are all shown to be unnecessary. In addition, we utilize our framework to reveal and analyze a novel primal--dual relationship between backward--backward splitting and a special instance of the alternating minimization algorithm. This illustrates how our duality framework provides a principled basis for primal--dual algorithm design and analysis in nonmonotone settings. Finally, our experiments show that our backward--backward splitting--based approach reliably outperforms the FISTA and SR3 algorithms for the LASSO problem.
    
\end{abstract}

\section{Introduction}

We are motivated by \textit{primal problems} of the form
\begin{equation}
\minimize_{(x,z)\in\H_x\times\H_z}\quad f(x)+g(z) \quad \subjectto \quad Ax+Bz+d=0,
\label{primal_optimisation_problem}
\end{equation}
where $\H_x$, $\H_z$ and $\H$ are real Hilbert spaces, $f:\H_x\to(-\infty,+\infty]$ and $g:\H_z\to(-\infty,+\infty]$ are proper lower semicontinuous and possibly nonconvex functions, $A:\H_x\rightarrow \H$, $B:\H_z\rightarrow \H$ are bounded linear operators and $d\in\H$. Problems of the form \eqref{primal_optimisation_problem} regularly appear across many disciplines, including statistics, engineering, physics and finance \cite{lindstrom2021survey}. 

Problem \eqref{primal_optimisation_problem} is commonly addressed by exploiting problem duality and employing primal--dual algorithms. In the case that $f$ and $g$ are also convex, the classical and well--established notion of duality is that of Fenchel and Rockafellar \cite{rockafellar1970convexanalysis}, which relies upon the connection between a proper lower semicontinuous and convex function $h:\H\to (-\infty,+\infty]$ and its Legendre--Fenchel transform $h^*(u)\coloneqq \sup_x \langle u , x \rangle - h(u)$. This relationship is exploited and forms the basis for several popular and effective primal--dual splitting methods \cite{ScottAndrew,ChambollePock,MoursiZinchenko,poon2019trajectory,tseng1991fbsduality}. 

In the context of Fenchel--Rockafellar duality, Daniel Gabay uncovered that the popular alternating direction method of multipliers (ADMM) applied to \eqref{primal_optimisation_problem} is equivalent to the Douglas--Rachford method (DR) applied to its associated dual problem \cite{gabay1983chapter}.
This seminal discovery has led to improved understanding of both methods, with ADMM convergence results often derived through analysis of DR \cite{EcksteinYao,GiselssonLinear,yuan2012convergence,themelis_2022}.
We refer the interested reader to \cite[Appendix 5]{lindstrom2021survey} for an extended synopsis of Gabay's original result.
It was later revealed by Paul Tseng that a connection similar to that of ADMM and DR is shared between forward--backward splitting (FBS) and the alternating minimization algorithm (AMA) \cite{tseng1991fbsduality}.
The results of Gabay and Tseng rely upon Fenchel--Rockafellar duality and hence the properties and machinery of convex analysis.
We explore the connection shared between algorithms using a novel notion of duality, with our contributions summarized below.
In recognition of Gabay's seminal observation of the connection between ADMM and DR, we refer to our algorithmic duality framework as \emph{Gabay duality}.

\paragraph{Contributions}

This work establishes Gabay duality: a framework and notion of duality between resolvent splitting methods and a general alternating primal--dual class of algorithms that holds in settings where classical duality breaks down. We show in Theorem~\ref{thm:Moreaus_identity} that even with no problem assumptions, every generalized resolvent operator possesses a unique resolvent to which it is linked in a meaningful and explicitly describable way. We call this unique resolvent its Gabay dual. Using Gabay duals, we show in Theorems \ref{thm:AMA_duality} and \ref{thm:ADMM_duality} that the classical relationships between FBS and AMA and between DR and ADMM hold even in the general nonconvex setting.
Interestingly, the same advantage that makes these results fully general---their non-dependency on the machinery of convex analysis---also makes them the simplest in the literature.
We furthermore provide explicit identities linking the sequences of these methods, immediately allowing practitioners to apply duality--based extrapolation schemes, such as dynamical systems approaches \cite{lindstrom2022computable,poon2019trajectory}, to nonconvex problems.
We further illustrate the usefulness of Gabay duality by deriving and analyzing a powerful primal algorithmic approach whose dual is backward--backward splitting. These contributions, along with the ways in which they subsume classical theory, are summarized in Tables~\ref{tab:problem_duality}, \ref{tab:resolvent_dualities}, and \ref{tab:algorithm_dualities}, whose notation we will next introduce. Finally, we numerically benchmark our primal backward--backward splitting approach on both convex and nonconvex instances of regression problems, by comparing its performance to the popular fast proximal--gradient method and SR3 algorithms \cite{beck2009fista,zheng2019sr3}. We observe that our approach reliably and significantly outperforms both fast proximal--gradient method and SR3.

\bigskip

We say the tuple $(x,z)$ is a stationary point of \eqref{primal_optimisation_problem} if there exists $\lambda\in\H$ such that $(x,z,\lambda)$ is a solution to the following problem: 
\begin{equation}
    \text{Find}\quad (x,z,\lambda) \in \H_x \times \H_z \times \H \quad \text{such that} \quad 0\in
    \begin{bmatrix}
    F(x)+A^*\lambda\\
    G(z)+B^*\lambda\\
    Ax+Bz+d
    \end{bmatrix},
    \label{primal_inclusion_problem}\tag{$\ensuremath{\mathcal{P}}$-inc}
\end{equation}
where $F$ and $G$ are the potentially set--valued subdifferential operators associated with $f$ and $g$, respectively. Problem \eqref{primal_inclusion_problem} forms the basis of our duality system, capturing both nonconvex and general operator settings. We associate to the \emph{primal inclusion problem} \eqref{primal_inclusion_problem} the \emph{dual inclusion problem}:
\begin{equation}
\text{Find}\;\; \lambda \in \H \quad \text{such that} \quad 0\in D_F(\lambda)+D_G(\lambda),
\label{dual_inclusion_problem}\tag{$\ensuremath{\mathcal{D}}$-inc}
\end{equation}
where we have 
$$
D_F\coloneqq (-A\circ F^{-1}\circ -A^*)(\cdot)-d_F,\quad D_G\coloneqq(-B\circ G^{-1}\circ -B^*)(\cdot)-d_G,
$$
and $d_G+d_F=d$.\footnote{Typically simplified as $(d_F,d_G) \in \{ (d,0),(0,d)\}$---we retain the ambiguity for complete generality of our derivations.} We describe the connection between problems \eqref{primal_inclusion_problem} and \eqref{dual_inclusion_problem} in Theorem~\ref{prop:primal_dual_problems}, after which we discuss how our notion of duality subsumes both Fenchel--Rockafellar duality and Attouch--Th\'{e}ra duality \cite{attouch1996duality}, where the latter relates the zeros of the sum of two operators $J,K:\H\rightrightarrows\H$ to the zeros of $J^{-1}$ and $K^{-\otimes}$, where $K^{-\otimes}\coloneqq -I\circ K^{-1}\circ -I$. Table~\ref{tab:problem_duality} summarizes and contextualizes our contribution, together with how it subsumes the classical theory.

\begin{table}[htb!]
    \centering%
\begin{tabular}{c p{4.03cm} c}
    \textbf{Primal} && \textbf{Dual} \\
    \midrule
    \shortstack{$\displaystyle \minimize_{(x,z)}\; f(x)+g(z)$ subject \\ to $Ax+Bz+d_F+d_G=0$}
    &
    \raisebox{2.5ex}{
    $\;\;\;\xLeftrightarrow[\text{(special case of Thm.~\ref{prop:primal_dual_problems})}]{\text{(convex case \cite{rockafellar1970convexanalysis})}}$}
    &
    \shortstack{$\displaystyle \minimize_{\lambda}\; f^*(-A^*\lambda)+$\\ $g^*(-B^*\lambda)-\langle \lambda,d_F+d_G\rangle$ \color{white}I$^I$\color{black}} \\
    &&\\
    \shortstack{find $y$
    such that \\ $0\in J(y)+K(y)$}
    &
    \raisebox{1.5ex}{
    $\;\;\;\xLeftrightarrow[\text{(special case of Thm.~\ref{prop:primal_dual_problems})}]{\text{(monotone case \cite{attouch1996duality})}}$}
    &
    \shortstack{find $\lambda$ such that \\ $0\in J^{-1}(\lambda)+K^{-\otimes}(\lambda)$.}\\
    &&\\
    \shortstack{find $(x,z,\lambda)$
    such that \\ $0\in
    \begin{bmatrix}
    F(x)+A^*\lambda\\[1pt]
    G(z)+B^*\lambda\\[1pt]
    Ax+Bz+d_F+d_G
    \end{bmatrix}$}
    &
    \raisebox{1.5ex}{
    $\;\;\;\xLeftrightarrow{\;\;\;\;\qquad\text{(Thm.~\ref{prop:primal_dual_problems})}\qquad\;\;\;\;}$}
    &
    \shortstack{find $\lambda$ such that \\ $0\in D_F(\lambda)+D_G(\lambda)$.}
\end{tabular}
\caption{Relevant primal--dual relationships.}
\label{tab:problem_duality}
\end{table}

We establish in Theorem~\ref{thm:Moreaus_identity} a connection between the \emph{primal operators} $F$ and $G$ present in \eqref{primal_inclusion_problem} and the \emph{dual operators} $D_F$ and $D_G$ arising in \eqref{dual_inclusion_problem}. The connection we show generalizes the classical Moreau proximal identity linking the proximal operators of functions and their Legendre--Fenchel transforms in the Fenchel--Rockafellar setting \cite{Moreau1965}. However, our result relaxes all assumptions on the set--valued mappings and linear operators involved, constituting a complete (maximally general) identity for pairing the operators in \eqref{primal_inclusion_problem} and \eqref{dual_inclusion_problem}. For this reason, we refer to Theorem~\ref{thm:Moreaus_identity} as the \emph{complete Moreau's identity}. This identity constitutes the basis of our analysis. The relationships between primal and dual mappings in both the classical Fenchel--Rockafellar sense and in our maximal sense are illustrated in Table~\ref{tab:resolvent_dualities}.

\begin{table}[htb]
    \centering%
\begin{tabular}{c p{4.03cm} c}
    \textbf{Primal} && \textbf{Dual} \\
    \midrule
    $\prox_h$
    &
    $\;\;\;\xLeftrightarrow[\text{(special case of Thm.~\ref{thm:Moreaus_identity})}]{\text{(convex case \cite{Moreau1965})}}$
    &
    $\prox_{h^*}$ \\[6pt]
    &&\\
    $(M^*M+H)^{-1}$
    &
    $\;\;\;\xLeftrightarrow{\,\;\;\qquad\text{(Thm. \ref{thm:Moreaus_identity})}\,\;\;\qquad}$
    &
    $(I+D_H)^{-1}$\\
\end{tabular}
\caption{Proximal and resolvent Moreau identities.}
\label{tab:resolvent_dualities}
\end{table}

By relying on Fenchel--Rockafellar duality, algorithmic dualities such as those shown by Gabay and Tseng inherently rely on convexity assumptions.
Our work operates in the operator--theoretic setting and exploits the general relationship between problems \eqref{primal_inclusion_problem} and \eqref{dual_inclusion_problem}.
Subsequently, we find that the connection between the algorithms originally shown by Gabay and Tseng can be shown algebraically and hence are independent of problem structure. Importantly, this means that researchers can now use popular primal-dual extrapolation schemes like A3DMM \cite{poon2019trajectory} for nonconvex problems, both in the sense that governing dual algorithms are well defined and also in the sense that our work shows explicitly how to compute them. Accordingly, we use our Gabay duality framework to furnish the fully general primal-dual relationships between FBS and AMA, and between DR and ADMM, so that any practitioner can use our work as a dictionary for exploiting their primal-dual relationships to extrapolate. Importantly, our resolvent framework is not limited to these algorithmic pairings, or even to the two operator setting in which we present our work, but is fully adaptable for resolvent operator splitting schemes in general. For example, we use it to uncover a connection between \emph{backward--backward splitting} and a special instance of AMA, which we refer to as \emph{primal backward--backward splitting} (PBBS).
We summarize the three examplar algorithmic dualities in Table~\ref{tab:algorithm_dualities}, highlighting how this work extends the connections known in the convex optimization setting. 

\begin{table}[htb]
    \centering%
\begin{tabular}{c p{3.33cm} c}
    \textbf{Primal} && \textbf{Dual} \\
    \midrule
    ADMM
    &
    $\;\;\;\xLeftrightarrow[\text{(all cases Thm.~\ref{thm:ADMM_duality})}]{\text{(convex case \cite{gabay1983chapter})}}$
    &
    Douglas--Rachford method \\[-6pt]
    &&\\
    AMA
    &
    $\;\;\;\xLeftrightarrow[\text{(all cases Thm.~\ref{thm:AMA_duality})}]{\text{(convex case \cite{tseng1991fbsduality})}}$
    &
    forward--backward splitting\\[-6pt]
    &&\\
    PBBS
    &
    $\;\;\;\xLeftrightarrow{\text{(all cases Thm.~\ref{thm:PBBS_duality})}}$
    &
    backward--backward splitting
\end{tabular}
\caption{Algorithmic dualities.}
\label{tab:algorithm_dualities}
\end{table}

\paragraph{Outline}
The rest of this paper is organized as follows.
Section~\ref{sec:background} reviews basic concepts that will be used throughout.
In section~\ref{sec:Gabay_duality} we introduce Gabay duality and utilize it to extend the known connections between FBS and AMA and DR and ADMM in sections \ref{sec:FBSandAMA} and \ref{sec:DRandADMM}, respectively.
In section~\ref{sec:PBBS} we use our framework to derive and analyze a novel primal--dual approach, which we numerically validate for regression problems in section~\ref{sec:experiment_results}.

\section{Background and Preliminaries}\label{sec:background}
Here and throughout, $I$ denotes the identity mapping.
We denote by $\Fix T = \{x : Tx=x\}$ the fixed point set of a set-valued operator $T:\H \rightrightarrows\H$; $T^{-1}$ is the set-valued `inverse' map that satisfies $y \in Tx \iff x \in T^{-1}y$.
The graph of a potentially set-valued mapping $F$ is denoted by $\gr{F}:=\{(x,y):y\in F(x)\}$.
For sets $U$ and $V$ we define $U+V:=\{u+v:u\in U,\; v\in V\}$.
We have that $\Gamma_0(X)$ denotes the set of proper lower semicontinuous convex functions on $X$. Sequences are written $(x_k)_{k\in K}$ for an index set $K\subseteq \N_0$; we abbreviate this to $(x^k)_K$, and simply $(x^k)$ when $K=\N_0$. We first recall notions of generalized derivatives. 

\begin{definition}[Subdifferentials]
    Let $h:\H\rightarrow (-\infty,+\infty]$ be a proper and lower semicontinuous function.
    The Fr\'{e}chet subdifferential of $h$ at $x\in \H$ is given by
    $$
    \partial_{\mathrm{F}}h(x)\coloneqq \Big\{v\in \H :\underset{u\rightarrow x}{\lim\inf}\;\frac{h(u)-h(x)-\langle v,u-x\rangle}{\|u-x\|}\geq 0\Big\}.
    $$
    The Mordukhovich subdifferential of $h$ at $x\in \H$ is defined as
    $$
    \partial_{\mathrm{M}} h(x)\coloneqq \big\{x^*\in \H: \exists x_n\rightarrow x,\;x^*_n\in \partial_{\mathrm{F}}h(x_n)\text{ with } h(x_n)\rightarrow h(x),\;x_n^*\rightarrow x^*\big\}.
    $$
    Taking the convex hull of the Mordukhovich subdifferential yields the Clarke subdifferential, that is, $\partial h(x)\coloneqq \conv\{\partial_{\mathrm{M}}h(x)\}$. 
\end{definition}
Throughout we refer to the Clarke subdifferential simply as the subdifferential and denote it using $\partial$. 

If $h\in\Gamma_0(\H)$, the Fr\'{e}chet, Mordukhovich and Clarke subdifferentials coincide with the classical subdifferential: $\{v: h(y)\geq h(x)+\langle v,y-x\rangle \;\forall y\in \H\}$.
In such cases, $\partial h:\H\rightrightarrows \H$ is monotone: $\langle u - v,u' - v' \rangle \geq 0\; \forall (u,u'),(v,v')\in \gr{\partial h}$ and there exists no monotone operator whose graph strictly contains $\gr{\partial h}$.
The latter property implies that $\partial h$ is \emph{maximal} monotone. 

\begin{definition}[Gabay dual operator]\label{def:dual_operator}
    Let $H:\H_1\rightrightarrows\H_1$, $M:\H_1\rightarrow \H$, and $d\in\H$.
    We refer to 
    $$
    D_H\coloneqq(-M\circ H^{-1}\circ -M^*)(\cdot)-d
    $$
    as the \emph{Gabay dual operator} of $H$ with respect to $M$ with translation vector $d$.
\end{definition}
The dual relationship between \eqref{primal_inclusion_problem} and \eqref{dual_inclusion_problem} is described in the following theorem.
\begin{theorem}\label{prop:primal_dual_problems}
    Let $\lambda\in\H$. The following are equivalent: 
    \begin{enumerate}[label=(\roman*)]
        \item $0\in D_F(\lambda)+D_G(\lambda)$;
        \item there exists $x\in F^{-1}(-A^*\lambda)$ and $z\in G^{-1}(-B^*\lambda)$ such that
    \begin{equation}
    0\in
    \begin{bmatrix}
    F(x)+A^*\lambda\\
    G(z)+B^*\lambda\\
    Ax+Bz+d_F+d_G
    \end{bmatrix}.\nonumber
    \end{equation}
    \end{enumerate}
\end{theorem}
\begin{proof}
    ($\implies$).
    Observe that $D_F(\lambda)=\{-Ax-d_F:x\in F^{-1}(-A^*\lambda)\}$ and $D_G(\lambda)=\{-Bz-d_G:z\in G^{-1}(-B^*\lambda)\}$.
    Suppose that \eqref{dual_inclusion_problem} holds.
    Then there exists $u_F\in D_F(\lambda)$ and $u_G\in D_G(\lambda)$ such that 
    $$
    0=u_F+u_G=(-Ax-d_F)+(-Bz-d_G)\iff 0= Ax+Bz+d_F+d_G.
    $$
    By how we have expressed $D_F(\lambda)$, we have that $-A^*\lambda\in F(x)$ as required.
    Similarly we also have $0\in G(z)+B^*\lambda$.

    ($\impliedby$).
    Suppose there exists $x\in F^{-1}(-A^*\lambda)$ and $z\in G^{-1}(-B^*\lambda)$.
    It follows that $-Ax-d_F\in D_F(\lambda)$ and $-Bz-d_G\in D_G(\lambda)$.
    The third condition of \eqref{primal_inclusion_problem} grants
    $$
    0\in Ax+Bz+d_F+d_F\iff 0\in (-Ax-d_F)+(-Bz-d_G).
    $$
    Since $-Ax-d_F\in D_F(\lambda)$ and $-Bz-d_G\in D_G(\lambda)$, we have $0\in D_F(\lambda)+D_G(\lambda)$. 
\end{proof}

We note that the classical Attouch--Th\'{e}ra duality \cite{attouch1996duality} is a special case of Theorem~\ref{prop:primal_dual_problems} corresponding to the setting where $A=-I$, $B=-I$, and $d_F=d_G=0$. 

Theorem~\ref{prop:primal_dual_problems} also generalizes Fenchel--Rockafellar duality presented in the convex optimization setting (cf.\ Table~\ref{tab:problem_duality}) . If $F=\partial f$ and $G=\partial g$ for $f\in\Gamma_0(\H_x)$ and $g\in\Gamma_0(\H_z)$, and under mild conditions (e.g. \cite[Theorem 2.6.6]{Clarke:4}), problem \eqref{dual_inclusion_problem} generalizes the Fenchel--Rockafellar dual of \eqref{primal_optimisation_problem}. This equivalence follows from the conjugacy relation $\partial h^*=(\partial h)^{-1}$ enjoyed by closed and convex $h$ \cite{bauschke2017monotoneoperators}. Under a standard constraint qualification \cite{rockafellar1970convexanalysis}, e.g. $0\in \mathrm{ri}(A(\dom f)+B(\dom g)+d_F+d_G)$, \eqref{primal_inclusion_problem} generalizes problem \eqref{primal_optimisation_problem}. This follows from the well--known fact that any stationary point $(x^*,z^*)$ is a solution to \eqref{primal_optimisation_problem}, and any solution to \eqref{primal_optimisation_problem} is necessarily a stationary point in this setting.

Problem \eqref{primal_optimisation_problem} is often addressed using methods based on Lagrangian duality, including the classical method of multipliers and ADMM.
ADMM is a standout example of the specific class of operator splitting methods we consider herein.

\begin{definition}[Alternating primal--dual class]
Let $\rho>0$ and $\Phi_x^{\rho},\Phi_z^{\rho}:\H_x\times\H_z\times\H\to(-\infty,+\infty]$.
We say that an algorithm belongs to the \emph{alternating primal--dual class} if it generates a sequence $(x^k,z^k,\lambda^k)$ according to the following updates of Gauss--Seidel style:
\begin{subequations}\label{AMA_expression}
\begin{align}
x^{k+1}&\in\argmin_{x\in\H_x}\quad \Phi_x^{\rho}(x,z^k,\lambda^k),\\
z^{k+1}&\in\argmin_{z\in\H_z}\quad \Phi_z^{\rho}(x^{k+1},z,\lambda^k),\\
\lambda^{k+1}&=\lambda^k+\rho(Ax^{k+1}+Bz^{k+1}+c).\label{APDC_dual_update}
\end{align}
\end{subequations}
\end{definition}
We refer to $(x^k)$ and $(z^k)$ above as our \emph{primal sequences} and $(\lambda^k)$ as a \emph{dual shadow sequence}, the latter nomenclature being standard in operator splitting literature for reasons made clear in section~\ref{sec:Gabay_duality}.
ADMM is recovered with the special choice
$$
\Phi_x^{\rho}(x,z,\lambda)=\Phi^{\rho}_z(x,z,\lambda)=\mathcal{L}_{\rho}(x,z,\lambda)\coloneqq f(x)+g(z)+\langle\lambda,Ax+Bz+c\rangle+\frac{\rho}2\|Ax+Bz+c\|^2.
$$
In this case $\Phi^{\rho}_x$ and $\Phi_z^{\rho}$ are the \emph{augmented Lagrangian} corresponding to \eqref{primal_optimisation_problem}.
AMA is another well--established member of the alternating primal--dual class, in which case $\Phi_x^{\rho}=\mathcal{L}_0$ and $\Phi_z^{\rho}=\mathcal{L}_{\rho}$.
For both of these methods, it is known in the convex setting that the dual update \eqref{APDC_dual_update} is related to a sequence generated via
\begin{equation}
y^{k+1}\in T(y^k)
\label{fixed_point_iteration}
\end{equation}
for some $T:\H\rightrightarrows\H$ such that $y\in\Fix T$ allows us to recover a solution to \eqref{dual_inclusion_problem}.
We refer to methods of style \eqref{fixed_point_iteration} as \emph{fixed--point iterations}.
For example, we will see in section~\ref{sec:DRandADMM} that for ADMM, dual sequence \eqref{APDC_dual_update} corresponds to the \emph{shadow} of \eqref{fixed_point_iteration} when $T$ is the DR operator.
We work to show throughout that when $T$ enjoys certain structure, sequence \eqref{fixed_point_iteration} applied to \eqref{dual_inclusion_problem} naturally induces an alternating primal--dual algorithm for addressing \eqref{primal_inclusion_problem} (or a related problem).
We show that this occurs when $T$ is expressed using resolvent operators, whose definition we now recall.
\begin{definition}[Resolvent]
    Let $H:\H\rightrightarrows \H$.
    The \emph{resolvent} of $H$ is the, generically set--valued and possibly empty, operator $J_{H}\coloneqq (I+H)^{-1}$.
\end{definition}
If $H$ is maximal monotone, it is well known that its resolvent is nonempty, single--valued, and \emph{firmly nonexpansive} \cite{bauschke2017monotoneoperators}: $\|J_{H}x-J_{H}y\|^2\leq \langle x-y,J_{ H}x-J_{H}y\rangle$ for all $x,y\in\dom(H)$.
The latter property is strictly stronger than \emph{nonexpansivity}: $\| J_Hx-J_Hy\| \leq \| x-y \|$ for all $x,y\in\dom(H)$. If $h\in\Gamma_0(\H)$ and $\gamma>0$, $J_{\gamma (\partial h)}$ agrees with the \emph{proximal operator} of $h$:
$$
J_{\gamma (\partial h)}(x)=\prox_{\gamma h}(x)\coloneqq \argmin_{y\in\H} \left\{ h(y)+\frac{1}{2\gamma}\big\|y-x\|^2 \right\}.
$$
Other useful cases where the resolvent is nonempty and single-valued include when it specifies to the proximity operator for a nonconvex sparse regularizer, such as the $\ell_0$ norm, minimax concave penalty (MCP), smoothly clipped absolute deviations (SCAD), or $\ell_p$ norms.
\begin{definition}[Generalized resolvent]\label{def:generalised_resolvents}
    Let $H:\H\rightrightarrows \H$ and $U:\H\rightarrow \H_1$.
    The \emph{generalized resolvent} of $H$ (with respect to $U$) is the mapping 
    $$
    J_{H}^U \coloneqq (U^* U + H)^{-1}.
    $$
\end{definition}
It is likewise straightforward to show that $J_H^U$ is nonempty, single-valued, and firmly nonexpansive whenever $H$ is maximally monotone and $U^*U$ is a bounded linear operator with full domain. The generalized resolvent captures the stationary points of subproblems that frequently appear in alternating primal--dual algorithms.
\begin{proposition}\label{prop:argmin_in_gen_resolvent}
    Let $H=\partial h$ for $h:\H\to(-\infty,+\infty]$ and $U:\H\rightarrow \H_1$. If $h$ is proper and lower semicontinuous, then
    \begin{equation}  \argmin_y\left\{ h(y)+\frac12\|Uy+v\|^2 \right\} \subseteq J_H^U( -U^*v), 
    \label{argmin_problem}
    \end{equation}
    with equality when $h$ is, in addition, convex.
\end{proposition}
\begin{proof}
    Suppose $y$ is a global minimizer of $y\mapsto h(y)+\frac12\|Uy+v\|^2$, from which it follows that it is necessarily also a local minimizer.
    We then have by \cite[Theorem~10.1]{rock:6} that $0 \in \partial_M \left(h(\cdot) + \frac{1}{2}\|U(\cdot)+v\|^2\right)(y)$.
    Since $U$ is bounded and $\frac12\|U(\cdot)+v\|^2$ is $C^1$, we have that $ \partial_M \left(h(\cdot) + \frac12\|U(\cdot)+v\|^2 \right)(y) = \partial_M h(y) + U^*(Uy+v)$ \cite[Proposition~1.107]{M06a}.
    Then, from the fact that $\partial_M h(y) \subseteq \operatorname{conv} \partial_M h(y) =\partial h(y)$, it follows that $0 \in \partial h(y) + U^*(Uy+v).$ Upon rearrangement, we obtain $y \in (U^*U + \partial h)^{-1}(-U^*v)$, as claimed.
\end{proof}

\begin{remark}\label{rem:gen_resolvent_is_argmin}
Even when $h$ is possibly nonconvex, authors will often, by a slight abuse of notation, use the left hand side's `argmin' representation in \eqref{argmin_problem} to represent generalized resolvent updates of form $J_H^U(-U^*v)$. We will adopt this convention in sections \ref{sec:Gabay_duality} and \ref{sec:PBBS}.
\end{remark}

\section{Gabay duality}\label{sec:Gabay_duality}
The algorithmic duality framework we propose relies upon a connection between the \emph{dual resolvent} of a Gabay dual operator $\smash{J_{\gamma D_H}}$ and the generalized resolvent $\smash{J^M_{\gamma^{-1}H}}$. This connection is maximally general and subsumes Moreau's identity in the operator setting.
\begin{theorem}[Complete Moreau identity]\label{thm:Moreaus_identity}
    Let $H:\H_1\rightrightarrows\H_1$, $M:\H_1\rightarrow\H$ be linear, $d\in\H$ and $\gamma >0$.
    Then
    \begin{equation}\tag{CMI}
    J_{\gamma D_H}(u) = u + \gamma d + \gamma M J^M_{\gamma^{-1} H}(-M^*(\gamma^{-1}u +d))
    \qquad\forall u\in\H.
    \label{CMI}
    \end{equation}
    This identity holds even when the operators $J_{\gamma D_H}$ and $J^M_{\gamma^{-1} H}$ do not have full domain, i.e., may have empty images, wherefore $J_{\gamma D_H}$ has full domain whenever the primal generalized resolvent $J_{\gamma^{-1}H}^M$ does.
    Furthermore, if $\lambda\in J_{\gamma D_H}(u)$, there exists some $q\in \H$ such that the following hold:
    \begin{enumerate}[label=(\roman*)]
        \item $q\in H^{-1}(-M^*\lambda)$ and $u-\lambda = -\gamma(Mq+d)$; \label{CMI_item1}
        \item $q\in J^M_{\gamma^{-1} H}(-M^*(\gamma^{-1}u+d))$. \label{CMI_item2}
    \end{enumerate}
\end{theorem}
We say $J_{\gamma D_H}$ is the \textit{Gabay dual resolvent} to the generalized resolvent $J_{\gamma^{-1}H}^M$ with translation $d$.
\begin{proof}
    The proof relies on straightforward manipulations of set--valued operators and linear maps.
    To show the result, it suffices to fix any $u \in \H$ and show the set equality:
    $$
    (I+\gamma D_H)^{-1}(u) = u + \gamma d + \gamma M(M^*M+\gamma^{-1}H)^{-1}(-M^*(\gamma^{-1} u +d)).
    $$
    This is what we will do.
    
    ($\subseteq$).
    The result is vacuous if $J_{\gamma D_H}(u) = \emptyset$.
    Let $\lambda \in J_{\gamma D_H}(u)$, which implies $u=\gamma p+\lambda$ for some $p \in D_H(\lambda)$.
    Furthermore, $p\in D_H(\lambda)$ implies the existence of $q \in H^{-1}(-M^*\lambda)$ such that $p=-Mq-d$.
    This shows \ref{CMI_item1}.
    It follows that
    \begin{align*}
        q \in H^{-1}(-M^*\lambda) \quad\iff& \quad Hq \ni-M^* \lambda \\
        \iff& \quad (M^*M+\gamma^{-1}H)q \ni M^*Mq -\gamma^{-1}M^* \lambda\\
        \iff& \quad q \in (M^*M+\gamma^{-1}H)^{-1}(-M^*(\gamma^{-1}\lambda-Mq))\\
        \iff& \quad q\in (M^*M+\gamma^{-1}H)^{-1}(-M^*(\gamma^{-1}\lambda+p+d))\\
        \iff& \quad q\in (M^*M+\gamma^{-1}H)^{-1}(-M^*(\gamma^{-1}\lambda+\gamma^{-1}(u-\lambda)+d))\\
        \iff& \quad q\in (M^*M+\gamma^{-1}H)^{-1}(-M^*(\gamma^{-1}u+d)),
    \end{align*}
    with the final line showing \ref{CMI_item2}. Altogether we have
    \begin{align*}
        \lambda &= u - \gamma p = u-\gamma(-Mq-d) \in u + \gamma d+\gamma M(M^*M+\gamma^{-1}H)^{-1}(-M^*(\gamma^{-1}u+d)).
    \end{align*}
    This completes the proof of ($\subseteq$).

    ($\supseteq$).
    If $u+\gamma d+\gamma M(M^*M+\gamma^{-1}H)^{-1}(-M^*(\gamma^{-1}u+d))=\emptyset$, then the result is vacuous.
    Otherwise, there exists $q \in (M^*M+\gamma^{-1}H)^{-1}(-M^*(\gamma^{-1}u+d))$. Thus:
    \begin{align*}
        q \in (M^*M+\gamma^{-1}&H)^{-1}(-M^*(\gamma^{-1}u+d)) \\
        \iff& \quad  \gamma^{-1}Hq + M^*Mq \ni -M^*(\gamma^{-1}u+d)\\
        \iff&\quad Hq\ni -M^*(u+\gamma d)-\gamma M^*Mq\\
        \iff&\quad  q \in H^{-1}(-M^*(u+\gamma d+\gamma Mq))\\
        \iff&\quad -\gamma Mq-\gamma d \in -\gamma M(H^{-1}(-M^*(u+\gamma d+\gamma Mq)))-\gamma d\\
        \iff& \quad  -\gamma Mq-\gamma d\in \gamma D_H(u+\gamma d+\gamma Mq)\\
        \iff& \quad -\gamma Mq-\gamma d + (u+\gamma d+\gamma Mq) \in (I+\gamma D_H)(u+\gamma d+\gamma Mq)\\
        \iff&\quad u\in (I+\gamma D_H)(u+\gamma d+\gamma Mq)\\
        \iff& \quad J_{\gamma D_H}(u) \ni u+\gamma d+\gamma Mq.
    \end{align*}
    This establishes the direction $(\supseteq)$, concluding the proof.
\end{proof}
\begin{remark}
    Specifying $M=-I$, $d=0$ and $H = \partial h$ for $h\in\Gamma_0(\H)$, one sees that \eqref{CMI} is a completion of Moreau's proximal identity \cite{Moreau1965}: 
    \begin{equation}\label{Moreau_identity}
    \prox_{\gamma h^*}(u) =u - \gamma \prox_{\gamma^{-1}h}(\gamma^{-1}u).
    \end{equation}
\end{remark}

We will see that items \ref{CMI_item1} and \ref{CMI_item2} provide a principled way of deriving algorithmic dualities.
More precisely, Theorem~\ref{thm:Moreaus_identity}\,\ref{CMI_item1} yields the linear relationships linking primal and dual sequences.
We refer to these relationships as \emph{Gabay identities}.
The fact that these sequences belong to generalized resolvents and hence correspond to stationary points to certain subproblems then follows from Theorem~\ref{thm:Moreaus_identity}\,\ref{CMI_item2} and Proposition~\ref{prop:argmin_in_gen_resolvent}.
We illustrate this general framework in the following section by applying it to establish the general connection between FBS and AMA.  

\subsection{A general Gabay duality for AMA}\label{sec:FBSandAMA}
We first use Gabay duality to establish the general relationship between AMA and FBS. FBS addresses \eqref{dual_inclusion_problem} according to 
\begin{equation}
\lambda^{k+1}=J_{\rho D_G}(\lambda^k-\rho D_F(\lambda^k)).
\label{FBS}
\end{equation}

\begin{theorem}[Equivalence of FBS and AMA]\label{thm:AMA_duality}
Let $F:\H_x\rightrightarrows\H_x$ and $G:\H_z\rightrightarrows\H_z$. Let $A:\H_x\to\H$ and $B:\H_z\to\H$ be linear and $d\in\H$. Then, addressing \eqref{dual_inclusion_problem} with \eqref{FBS} induces the sequence $(x^k,z^k,\lambda^k)$ such that
\begin{subequations}\label{AMA_general} 
\begin{align}
x^{k+1}&\in F^{-1}(-A^*\lambda^k),\label{gen_AMA_x_update}\\
z^{k+1}&\in J^B_{\rho^{-1}G}(-B^*(Ax^{k+1}+d+\rho^{-1}\lambda^k)),\label{gen_AMA_z_update}\\
\lambda^{k+1}&=\lambda^k+\rho (Ax^{k+1}+Bz^{k+1}+d).\label{gen_AMA_dual_update}
\end{align}
\end{subequations}

\end{theorem}
\begin{proof}
Let $\lambda^{k+1}\in J_{\rho D_G}(\lambda^k-\rho D_F(\lambda^k))$. From the definition of $D_F$,there necessarily exists some $x^{k+1}\in F^{-1}(-A^*\lambda^k)$ such that $\lambda^{k+1}\in J_{\rho D_G}(\lambda^k+\rho Ax^{k+1})$. By Theorem~\ref{thm:Moreaus_identity}\,\ref{CMI_item1}, there exists $z^{k+1}\in G^{-1}(-B^*\lambda^{k+1})$ such that 
\begin{equation}
\lambda^k+\rho Ax^{k+1}-\lambda^{k+1}=-\rho(Bz^{k+1}+d).
\label{alg_equiv_eq1}
\end{equation}
We have by Theorem~\ref{thm:Moreaus_identity}\,\ref{CMI_item2}
that $z^{k+1}$ also satisfies

$$
z^{k+1}\in J^B_{\rho^{-1}G}(-B^*(\rho^{-1}\lambda^k+Ax^{k+1}+d)).
$$
Simply rearranging \eqref{alg_equiv_eq1} shows \eqref{gen_AMA_dual_update}.
\end{proof}

\begin{remark}
If $F=\partial g$ and $G=\partial g$, updates \eqref{AMA_general} may be expressed as
\begin{subequations}\label{AMA}
\begin{align}
x^{k+1}&\in\argmin_{x\in\H_x}\quad f(x)+\langle\lambda^k,Ax\rangle,\label{AMA_x_update}\\
z^{k+1}&\in\argmin_{z\in\H_z}\quad g(z)+\frac{\rho}{2}\left\|Ax^{k+1}+Bz+\rho^{-1}\lambda^k+d\right\|^2,\label{AMA_z_update}\\
\lambda^{k+1}&=\lambda^k+\rho\left( Ax^{k+1}+Bz^{k+1}+d\right).\label{AMA_dual_update} 
\end{align}
\end{subequations}
Method \eqref{AMA} is precisely AMA addressing \eqref{primal_optimisation_problem}, where the `argmin' is in the sense of Remark~\ref{rem:gen_resolvent_is_argmin}.
\end{remark}

\subsection{A general Gabay duality for ADMM}\label{sec:DRandADMM}

DR addresses \eqref{dual_inclusion_problem} under the typical assumption that $D_F,D_G:\H\rightrightarrows\H$ and $D_F+D_G$ are maximal monotone \cite{LionsMercier}.
The method generates a $(\lambda^k)$ sequence according to 
\begin{subequations}
\begin{align}
\lambda^{k+1}&\in J_{\rho D_G}(y^k),\label{dr_lambda_update}\\
\mu^{k+1}&\in J_{\rho D_F}(2\lambda^{k+1}-y^k),\label{dr_mu_update}\\
y^{k+1}&=y^k+\mu^{k+1}-\lambda^{k+1}.\label{dr_y_update}
\end{align}
\label{DRM}
\end{subequations}
Updates \eqref{DRM} can be expressed as a fixed point iteration of form \eqref{fixed_point_iteration} with
\begin{equation}
T=J_{\rho D_G}\circ (2J_{\rho D_F}-I)+I-J_{\rho D_F}.
\end{equation}

It is common in practice to attempt to solve \eqref{primal_optimisation_problem} by approaching \eqref{dual_inclusion_problem} with DR.
This is less common if \eqref{primal_optimisation_problem} is nonconvex, since solution equivalence may no longer hold and DR often lacks convergence guarantees.
It is known in the convex setting that the DR sequence $(\lambda^k,\mu^k,y^k)$ is linked to the ADMM sequence $(x^k,z^k,\lambda^k)$ by the Gabay identites \eqref{GI1}--\eqref{GI4}, see \cite{lindstrom2022computable}.
We show that these identities also hold in general in the following lemma. 
\begin{lemma}[ADMM and DR Gabay identities]\label{lem:Gabay_identities}
    Let $F:\H\rightrightarrows\H_x$ and $G:\H_z\rightrightarrows\H_z$.
    Let $A:\H_x\to\H$ and $B:\H_z\to\H$ be linear and $d_F,d_G\in\H$.
    Then updates \eqref{DRM} for \eqref{dual_inclusion_problem} induce the sequence $(x^k,z^k,\lambda^k)$ such that the following identities hold for all $k\geq 0$. 
    \begin{align}
    \lambda^{k+1}-y^k&=\rho(Bz^{k+1}+d_G),\label{GI1}\tag{G1}\\
    2\lambda^{k+1}-y^k&=\lambda^{k+1}+\rho(Bz^{k+1}+d_G),\label{GI2}\tag{G2}\\
    \mu^k-\lambda^k&=\rho(Ax^{k+1}+Bz^k+d_F+d_G),\label{GI3}\tag{G3}\\
    y^k-\lambda^k&=\rho(Ax^{k+1}+d_F).\label{GI4}\tag{G4}
    \end{align}
\end{lemma}
\begin{proof}
    By Theorem~\ref{thm:Moreaus_identity}\,\ref{CMI_item1}, we immediately see from \eqref{dr_lambda_update} that there exists some $z^{k+1}\in G^{-1}(-B^*\lambda^{k+1})$ satisfying \eqref{GI1}.
    Adding $\lambda^{k+1}$ to both sides of \eqref{GI1} yields \eqref{GI2}.
    Similarly applying Theorem~\ref{thm:Moreaus_identity}\,\ref{CMI_item1} to \eqref{dr_mu_update} yields the existence of $x^{k+2}\in F^{-1}(-A^*\mu^{k+1})$ such that 
    \begin{equation}
    2\lambda^{k+1}-y^k-\mu^{k+1}=-\rho(Ax^{k+2}+d_F).
    \label{almost_GI3}
    \end{equation}
    Substituting \eqref{GI2} into \eqref{almost_GI3} yields \eqref{GI3}.
    Finally, substituting \eqref{GI1} and \eqref{GI4} into \eqref{dr_y_update} yields
    \begin{multline*}
    y^{k+1}=\lambda^{k+1}-\rho(Bz^{k+1}+d_G)+\lambda^{k+1}+\rho(Ax^{k+2}+Bz^{k+1}+d_F+d_G)-\lambda^{k+1}\\
    \iff\quad y^k=\lambda^k+\rho(Ax^{k+1}+d_F).
    \end{multline*}
\end{proof}

We extend the known ADMM and DR duality in Theorem~\ref{thm:ADMM_duality}, wherein we show that DR applied to \eqref{dual_inclusion_problem} induces a primal algorithm akin to ADMM for \eqref{primal_inclusion_problem}.

\begin{theorem}[Equivalence of DR and ADMM]\label{thm:ADMM_duality}
    Let $F:\H_x\rightrightarrows\H_x$ and $G:\H_z\rightrightarrows\H_z$.
    Let $A:\H_x\rightarrow \H$ and $B:\H_z\rightarrow\H$ be linear and $d_F,d_G \in\H$.
    The DR updates \eqref{DRM} induce a sequence $(x^k,z^k,\lambda^k)$ satisfying
    \begin{subequations}
    \label{primal_alg_updates}
    \begin{align}
            x^{k+1}&\in J^A_{\rho^{-1}F}\left( -A^*(Bz^k+\rho^{-1}\lambda^k+d_{F}+d_{G})\right) , \label{primal_alg_x_update}\\
            z^{k+1}&\in J^B_{\rho^{-1}G}\left( -B^*(Ax^{k+1}+\rho^{-1}\lambda^k+d_{F}+d_{G})\right) , \label{primal_alg_z_update}\\
            \lambda^{k+1}&=\lambda^k+\rho(Ax^{k+1}+Bz^{k+1}+d_{F}+d_{G}).\label{primal_alg_dual_update}
        \end{align}
    \end{subequations}
\end{theorem}
\begin{proof}
From \eqref{dr_lambda_update} and Theorem~\ref{thm:Moreaus_identity}\,\ref{CMI_item2}, we have that there necessarily exists some $z^{k+1}\in G^{-1}(-B^*\lambda^{k+1})$ such that 
\begin{equation}
z^{k+1}\in J_{\rho^{-1}G}^B(-B^*(\rho^{-1}y^k+d_G)).
\label{almost_ADMM_z_update}
\end{equation}
Update \eqref{primal_alg_z_update} follows by substituting \eqref{GI4} into \eqref{almost_ADMM_z_update}. Applying Theorem~\ref{thm:Moreaus_identity}\,\ref{CMI_item2} to \eqref{dr_mu_update} grants the existence of $x^{k+2}\in F^{-1}(-A^*\mu^{k+1})$ such that
\begin{equation}
x^{k+2}\in J_{\rho^{-1}F}^A(-A^*(\rho^{-1}(2\lambda^{k+1}-y^k)+d_F)).
\label{almost_ADMM_x_update}
\end{equation}
Now combining \eqref{almost_ADMM_x_update} with \eqref{GI2} and taking $k\gets k-1$ yields \eqref{primal_alg_x_update}. Update \eqref{primal_alg_dual_update} follows by adding \eqref{GI1} and \eqref{GI4}.
\end{proof}

\begin{remark}
    Supposing that $F$ and $G$ from Theorem~\ref{thm:ADMM_duality} correspond to $\partial f$ and $\partial g$, we have by Proposition~\ref{prop:argmin_in_gen_resolvent} and Remark~\ref{rem:gen_resolvent_is_argmin} that \eqref{primal_alg_updates} agrees with classical ADMM:
    \begin{subequations}\label{ADMM}
    \begin{align}
    x^{k+1}&\in \argmin_{x\in\H_x} \quad f(x)+\frac{\rho}{2}\big\|Ax+Bz^k+\rho^{-1}\lambda^k+d_F+d_G\big\|^2 , \label{ADMM_x_update}
    \\
    z^{k+1}&\in \argmin_{z\in\H_z} \quad g(z)+\frac{\rho}{2}\Big\|Ax^{k+1}+Bz+\rho^{-1}\lambda^k+d_F+d_G\big\|^2, \label{ADMM_z_update}
    \\
    \lambda^{k+1}&=\lambda^k+\rho\left( Ax^{k+1}+Bz^{k+1}+d_F+d_G\right). \label{ADMM_dual_update}
    \end{align}
    \end{subequations}
\end{remark}

Theorem~\ref{thm:ADMM_duality} demonstrates that---provided the updates are defined---Gabay's duality--like relationship always holds, regardless of the properties of $F$, $G$ or the linear maps $A$ and $B$.
Notably, the general connection between DR and ADMM is shown to hold when $A$ and $B$ are not injective.
This allowance is important in stochastic programming, where the linear operators in question (often described implicitly) are generally not injective but consist of blocks of identity and zero matrices that enforce consensus for linked decision variables.

Lemma~\ref{lem:Gabay_identities} reveals that at each iteration $k$, a primal algorithm tuple $(x^k,z^k,\lambda^k)$ is computed at no added cost when DR is applied to \eqref{dual_inclusion_problem}.
This relationship enables methods that extrapolate on the sequence $(y^k)$, such as those in \cite{ScottAndrew,lindstrom2022computable,poon2019trajectory}.
An immediate application of Theorem~\ref{thm:ADMM_duality} is that such extrapolation schemes may now be applied to nonmonotone inclusion problems.

The following corollary links the DR sequences $(\lambda^k,\mu^k)$ with the standard ADMM residuals.
This enables convergence results obtained for the DR sequences to be immediately transferred to ADMM.
\begin{corollary}\label{cor:ADMM_residuals}
Suppose the DR sequences satisfy 
$$
\|\lambda^k-\mu^k\|\to0\quad\mand\quad\|\lambda^{k+1}-\lambda^k\|\to0.
$$
Then the ADMM \emph{primal} and \emph{dual} residuals $r^k= Ax^k+Bz^k+d_F+d_G$ and $s^k= \rho A^*B(z^k-z^{k-1})$ likewise vanish.
\end{corollary}
\begin{proof}
It is clear from \eqref{primal_alg_dual_update} that $\rho^{-1}\|\lambda^{k+1}-\lambda^k\|=\|r^{k+1}\|$.
Theorem~\ref{thm:ADMM_duality} and Lemma~\ref{lem:Gabay_identities} combine to show that the DR and ADMM sequences satisfy \eqref{GI1}--\eqref{GI4}. Subtracting \eqref{GI4} from \eqref{GI3} yields
\begin{equation}
\mu^k-y^k=\rho (Bz^{k+1}+d_G).
\label{almost_dual_residual}
\end{equation}
We have from subtracting \eqref{almost_dual_residual} from \eqref{GI1} that 
$$
\lambda^{k+1}-\mu^k=\rho B(z^{k+1}-z^k)\quad\implies\quad A^*(\lambda^{k+1}-\mu^k)=\rho A^*B(z^{k+1}-z^k).
$$
\end{proof}
In light of Corollary~\ref{cor:ADMM_residuals}, one finds that the differences $\|\lambda^k-\mu^k\|$ and $\|\lambda^{k+1}-\lambda^k\|$ vanishing, as is often observed in certain nonconvex settings, imply that the ADMM residuals converge to zero in norm.

\section{A primal--dual approach derived from Gabay duality}\label{sec:PBBS}
We now utilize our Gabay duality framework to derive and analyze a practical optimization approach.
We are motivated by certain statistical and inverse problems where a controlled approximation error to improve numerical tractability and performance is acceptable, if not desirable.
We propose addressing \eqref{primal_optimisation_problem} via the following quadratic relaxation:
\begin{equation}
\minimize_{(x,z)\in\H_x\times\H_z}\quad f(x)+g(z)+\frac{\rho}2\|Ax\|^2\quad\subjectto\quad Ax+Bz+d=0,
\label{PBBS_problem}
\end{equation}
where $A:\H_x\to\H$, $B:\H_z\to\H$ and $d\in\H$.

Problem \eqref{PBBS_problem} often naturally arises in  statistics and machine learning, since it may inherently incorporate Tikhonov regularization, that is, the integration of quadratic function $\frac{\rho}2\|x\|^2$ in the objective to improve problem conditioning \cite{tikhonov1943regularisation}.

We address \eqref{PBBS_problem} using AMA, in which case we show that AMA is dual to backward--backward splitting (BBS).

We exploit this connection to show that the approach of relaxing \eqref{primal_optimisation_problem} with \eqref{PBBS_problem} and applying AMA is provably convergent under standard convexity assumptions.
We refer to this approach as \emph{primal backward--backward splitting} (PBBS).

Importantly, since we impose no assumptions on $A$, the function $f(x)+\frac{\rho}2\|Ax\|^2$ is not necessarily strongly convex.
Our work thus presents a case where AMA may converge without strong convexity, extending the method's known convergence theory \cite{themelis2019ama,tseng1991fbsduality}.

Recalling the notational convention from Remark~\ref{rem:gen_resolvent_is_argmin}, AMA for \eqref{PBBS_problem} reads
\begin{subequations}\label{PBBS}
\begin{align}
x^{k+1}&\in\argmin_{x\in\H_x}\quad f(x)+\frac{\rho}2\|Ax+\rho^{-1}\lambda^k\|^2,\label{PBBS_x_update}\\
z^{k+1}&\in\argmin_{z\in\H_z}\quad g(z)+\frac{\rho}2\|Ax^{k+1}+Bz+d+\rho^{-1}\lambda^k\|^2,\label{PBBS_z_update}\\
\lambda^{k+1}&=\lambda^k+\rho(Ax^{k+1}+Bz^{k+1}+d).\label{PBBS_dual_update}
\end{align}
\end{subequations}

Applying Theorem~\ref{prop:primal_dual_problems}, one sees that the primal inclusion associated with \eqref{PBBS_problem} is the problem of locating $(x,z,\lambda)$ such that
\begin{equation}
0\in
\begin{bmatrix}
(\rho A^{\top}A+\partial f)x+A^*\lambda\\
\partial g(z)+B^*\lambda\\
Ax+Bz+d
\end{bmatrix}.
\label{PBBS_primal_inclusion}
\end{equation}
We now work to obtain the dual inclusion corresponding to \eqref{PBBS_problem}.
We first require the following definition.
\begin{definition}
Let $D:\H\rightrightarrows\H$.
The \emph{Yosida approximation} of $D$ with parameter $\alpha>0$ is given by
$$
D^{\alpha}\coloneqq \alpha^{-1}(I-J_{\alpha D}).
$$
\end{definition}

\begin{proposition}
The dual inclusion associated with \eqref{PBBS_problem} is that of locating $\lambda\in\H$ such that 
\begin{equation}
0\in D_F^{\rho}(\lambda)+D_G(\lambda),
\label{PBBS_dual_inclusion}
\end{equation}
where $D_F(\lambda)=-A(\rho A^*A+\partial f)^{-1}(-A^*\lambda))$ and $D_G(\lambda)=-B(\partial g^*(-B^*\lambda))-d$. 
\end{proposition}
\begin{proof}
The claimed dual operators corresponding to \eqref{PBBS_primal_inclusion} are clear from Theorem~\ref{prop:primal_dual_problems}.
Considering $D_F$, we have 
\begin{align*}
D_F(\lambda)=-A(A^*A+\rho^{-1}\partial f)^{-1}(-A^*(\rho^{-1}\lambda))
=\rho^{-1}(\lambda-J_{\rho D_F}(\lambda))
=D^{\rho}_F(\lambda),
\end{align*}
with the second equality following from \eqref{CMI} (Theorem~\ref{thm:Moreaus_identity}).
\end{proof}

A natural method to address \eqref{PBBS_dual_inclusion} is BBS, which updates according to 
\begin{subequations}
\begin{align}
\mu^{k+1}&\in J_{\rho D_F}(\lambda^k),\label{BBS_mu_update}\\
\lambda^{k+1}&\in J_{\rho D_G}(\mu^{k+1}).\label{BBS_lambda_update}
\end{align}
\label{BBS}
\end{subequations}
Updates \eqref{BBS} can be equivalently expressed as $\lambda^{k+1}\in T_{\lambda}(\lambda^k)$ with $T_{\lambda}\coloneqq J_{\rho D_G}\circ J_{\rho D_F}$.
It is also straightforward to see that we have $\mu^{k+1}\in T_{\mu}(\mu^k)$ with $T_{\mu}\coloneqq J_{\rho D_F}\circ J_{\rho D_G}$.
We will see that considering both $(\lambda^k)$ and $(\mu^k)$ sequences will be useful in studying PBBS.
The following theorem shows that AMA applied to \eqref{PBBS_problem} is equivalent to BBS applied to \eqref{PBBS_dual_inclusion}.
\begin{theorem}[Equivalence of BBS and PBBS]\label{thm:PBBS_duality}
Updates \eqref{BBS} induce a primal sequence $(x^k,z^k,\lambda^k)$ satisfying \eqref{PBBS}.
Furthermore, we have that 
$$
x^k\in(\partial f)^{-1}(-A^*\mu^k)\quad\mand\quad z^k\in(\partial g)^{-1}(-B^*\lambda^k)
$$
for all $k$.
\end{theorem}
\begin{proof}
From Theorem~\ref{thm:Moreaus_identity}\,\ref{CMI_item1} we have that \eqref{BBS_mu_update} induces some $x^{k+1}\in (\partial f)^{-1}(-A^*\mu^{k+1})$ satisfying 
\begin{equation}
\lambda^k-\mu^{k+1}=-\rho Ax^{k+1}.
\label{PBBS_identity}
\end{equation}
From \ref{CMI_item2} we have that 
$$
x^{k+1}\in J^A_{\rho^{-1}\partial f}(-\rho^{-1}A^*\lambda^k)=\argmin_{x\in\R^n}\quad f(x)+\frac{\rho}2\|Ax+\rho^{-1}\lambda^k\|^2,
$$
with the equality following from Proposition~\ref{prop:argmin_in_gen_resolvent}.
Now applying Theorem~\ref{thm:Moreaus_identity}\,\ref{CMI_item2} to \eqref{BBS_lambda_update}, we obtain the existence of some $z^{k+1}\in (\partial g)^{-1}(-B^{\top}\lambda^{k+1})$ such that
\begin{align*}
z^{k+1}&\in J^B_{\rho^{-1}\partial g}(-B^*(\rho^{-1}\mu^{k+1}+d))=J^B_{\rho^{-1}\partial g}(-B^*(\rho^{-1}\lambda^k+Ax^{k+1}+d)).
\end{align*}
The equality above follows from substituting \eqref{PBBS_identity}.
We have from applying Theorem~\ref{thm:Moreaus_identity}\,\ref{CMI_item1} to \eqref{BBS_lambda_update} that
\begin{equation}
\mu^{k+1}-\lambda^{k+1}=-\rho(Bz^{k+1}+d).
\label{PBBS_identity2}
\end{equation}
Adding \eqref{PBBS_identity} and \eqref{PBBS_identity2} yields claimed dual update \eqref{PBBS_dual_update}.
\end{proof}

\subsection{Convergence analysis}
For ease of presentation, we analyze the convergence of PBBS in the setting $\H_x=\R^n$ and $\H_z=\R^m$.
We now utilize Gabay duality to show PBBS sequence convergence in the convex setting.
\begin{assumption}\label{PBBS_assumptions}
Let problem~\eqref{PBBS_problem} satisfy the following.
\begin{enumerate}[label=(\roman*)]
    \item $f\in\Gamma_0(\R^n)$ and $g\in\Gamma_0(\R^m)$;\label{PBBS_function_assumption}
    \item $A\in\R^{r\times n}$, $B\in\R^{r\times m}$ and $d\in\R^r$;\label{PBBS_matrix_assumption}
    \item $\ran(-A^{\top})\cap \dom f^*\neq\emptyset$ and $\ran(-B^{\top})\cap \dom g^*\neq\emptyset$.\label{PBBS_set_assumption}
    \item $\Fix T_{\lambda}$ and $\Fix T_{\mu}$ are nonempty.\label{PBBS_fixed_point_assumption}
\end{enumerate}
\end{assumption}

\begin{proposition}
Let Assumption~\ref{PBBS_assumptions} hold. Then $D_{\partial f}$ and $D_{\partial g}$ are maximal monotone.
\end{proposition}
\begin{proof}
We work to show that Assumption~\ref{PBBS_assumptions} is sufficient to ensure that $f^*\circ -A^{\top}$ and $g^*\circ -B^{\top}$ are proper, lower semicontinous and convex.
The maximal monotonicity of $D_{\partial f}=\partial (f^*\circ -A^{\top})$ and $D_{\partial g}=\partial (g^*\circ -B^{\top})$ then follow from \cite[Theorem~20.25]{bauschke2017monotoneoperators}.
It is a well known result of Fenchel--Moreau that Assumption~\ref{PBBS_assumptions}\,\ref{PBBS_function_assumption} grants the properness, lower semicontinuity and convexity of $f^*$ and $g^*$ \cite[Theorem~13.37]{bauschke2017monotoneoperators}.
It is then straightforward to verify that the composition of a convex and lower semicontinuous function with a linear operator is convex and lower semicontinuous in finite dimensions.
By Assumption~\ref{PBBS_assumptions}\,\ref{PBBS_set_assumption} we have the existence of some $u_1$ and $u_2$ such that $f^*(-A^{\top}u_1)$ and $g^*(-B^{\top}u_2)$ are finite and hence proper.
\end{proof}

\begin{lemma}\label{lem:BBS_bounded}
    Suppose Assumption~\ref{PBBS_assumptions} holds.
    Then the $(\lambda^k)$ and $(\mu^k)$ sequences generated according to \eqref{BBS} are Fej\'{e}r monotone and asymptotically regular, i.e., 
    $$
    \|\lambda^{k+1}-\lambda^k\|\to0\quad\mand\quad \|\mu^{k+1}-\mu^k\|\to0.
    $$
\end{lemma}
\begin{proof}
    Since both sequences are generated similarly, we explicitly show the result for $(\lambda^k)$.
    Our argument may be applied to $(\mu^k)$ by taking $\lambda^k\gets \mu^k$ and $T_{\lambda}\gets T_{\mu}$.
    Since $T_{\lambda}$ is the composition of firmly nonexpansive mappings, it is $2/3$--averaged \cite[Example~4.45]{bauschke2017monotoneoperators}.
    Since $\Fix T_{\lambda}$ is nonempty by assumption, let $\lambda\in\Fix T_{\lambda}$.
    The averagedness of $T_{\lambda}$ grants that the following is satisfied for all $k$ \cite[Proposition~4.35(iii)]{bauschke2017monotoneoperators}:
    \begin{equation}
    \|\lambda^{k+1}-\lambda\|^2=\|T(\lambda^k)-T(\lambda)\|^2\leq \|\lambda^k-\lambda\|^2-\frac12\|\lambda^{k+1}-\lambda^k\|^2.
    \label{BBS_telescope}
    \end{equation}
    One sees that \eqref{BBS_telescope} implies that $(\lambda^k)$ is Fej\'{e}r monotone with respect to $\Fix T_{\lambda}$ and hence bounded \cite[Definition~5.1]{bauschke2017monotoneoperators}. Rearranging \eqref{BBS_telescope} and telescoping yields for any $K\geq 0$:
    $$
    \sum_{k=0}^K\|\lambda^{k+1}-\lambda^k\|^2\leq 2\|\lambda^0-\lambda\|^2,
    $$
    from which we have by taking $K\to\infty$ that $\|\lambda^{k+1}-\lambda^k\|$ is square summable and hence asymptotic regular.
\end{proof}

\begin{lemma}\label{lem:BBS_convergent}
    The BBS sequences $(\lambda^k)$ and $(\mu^k)$ are such that $\lambda^k\to\lambda\in\Fix T_{\lambda}$ and $\mu^k\to\mu\in\Fix T_{\mu}$.
\end{lemma}
\begin{proof}
    We once again show the result for $(\lambda^k)$ using arguments that may be repeated for $(\mu^k)$.
    Let $\lambda$ be a sequential cluster point of $(\lambda^k)$.
    By Lemma~\ref{lem:BBS_bounded} we have the existence of a subsequence $(\lambda^{k_j})$ such that $\lambda^{k_j}\to\lambda$.
    By the asymptotic regularity of $(\lambda^k)$ we also have $\lambda^{k_j+1}\to\lambda$.
    Since $\lambda^{k_j+1}=T(\lambda^{k_j})$ and $T_{\lambda}$ is continuous by merit of being nonexpansive, passing to the limit yields $T_{\lambda}(\lambda)=\lambda$.
    We therefore have that every sequential cluster point of $(\lambda^k)$ belongs to $\Fix T_{\lambda}$.
    The claim now follows from the Fej\'{e}r monotonicity of $(\lambda^k)$ and \cite[Theorem~5.5]{bauschke2017monotoneoperators}.
\end{proof}

\begin{lemma}\label{lem:bounded_bbs}
    Let $-A^{\top}\mu\in\mathrm{int}(\dom f^*)$ and $-B^{\top}\lambda\in\mathrm{int}(\dom g^*)$.
    Then the sequences $(x^k)$ and $(z^k)$ generated according to \eqref{PBBS} are bounded. 
\end{lemma}
\begin{proof}
    By Lemma~\ref{lem:BBS_convergent} we have that $\mu^k\to\mu\in\Fix T_{\mu}$.
    Since $f^*\in\Gamma_0(\R^n)$, $\partial f^*$ is locally bounded on $\mathrm{int}(\dom f^*)$ \cite[Theorem~24.7]{rockafellar1970convexanalysis}.
    Since we assume $-A^{\top}\mu\in\mathrm{int}(\dom f^*)$, there exists a neighborhood $\mathcal{N}$ of $-A^{\top}\mu$ and $\alpha>0$ such that $u\in\mathcal{N}$ and $x\in \partial f^*(u)$ implies $\|x\|\leq\alpha$.
    Since $-A^{\top}\mu^k\to -A^{\top}\mu$, there exists $K$ such that $-A^{\top}\mu^k\in\mathcal{N}$ for all $k\geq K$.
    From Theorem~\ref{thm:PBBS_duality} we have that $x^k\in\partial f^*(-A^{\top}\mu^k)$, from which it follows that $\|x^k\|\leq\alpha$ for all $k\geq K$.
    We may therefore conclude that the tail of $(x^k)$ is bounded and hence $(x^k)$ is bounded.
    A similar argument shows that $(z^k)$, which is such that $z^k\in\partial g^*(-B^{\top}\lambda^k)$, is also bounded.
\end{proof}

\begin{lemma}\label{lem:PBBS_convergent}
    Let $-A^{\top}\mu\in\mathrm{int}(\dom f^*)$ and $-B^{\top}\lambda\in\mathrm{int}(\dom g^*)$.
    Furthermore suppose that $\partial f^*(-A^*\mu)$ and $\partial g^*(-B^*\lambda)$ are singletons.
    The sequences $(x^k)$ and $(z^k)$ generated according to \eqref{PBBS} are such that
    $$
    x^k\to x\in\partial f^*(-A^{\top}\mu)\quad\mand\quad z^k\to z\in\partial g^*(-B^{\top}\lambda).
    $$ 
\end{lemma}
\begin{proof}
    Lemma~\ref{lem:BBS_convergent} and the linearity of $A$ grant that
    $$
    -A^{\top}\mu^{k_j+1}\to -A^{\top}\mu
    $$
    for any subsequence of $(\mu^k)$.
    We have that $(x^k)$ is bounded by Lemma~\ref{lem:bounded_bbs}, from which it follows that there exists cluster point $x$ and convergent subsequence $x^{k_j}$ such that $x^{k_j}\to x$. 
    Since we know $x^k\in\partial f^*(-A^{\top}\mu^k)$ to hold for all $k$ and by the graph closedness of $\partial f^*$, we have that 
    \begin{equation}
    x^{k_j}\in\partial f^*(-A^{\top}\mu^{k_j+1})\quad\implies\quad x\in \partial f^*(-A^{\top}\mu)
    \label{bbs_graph_closed}
    \end{equation}
    by passing to the limit.
    By assumption we have that $\partial f^*(-A^{\top}\mu^*)$ is a singleton, and so it follows that $x^k\to x\in \partial f^*(-A^{\top}\mu)$.
    An argument analogous to that above grants the convergence of $z^k\to z$ such that $z\in\partial g^*(-B^{\top}\lambda)$.
\end{proof}

\begin{proposition}[PBBS residual]\label{prop:PBBS_residuals}
    The residuals $s^k=\rho A^{\top}(Ax^k+Bz^k+d)$ and $r^k=Ax^k+Bz^k+d$ measure the PBBS update's $(x^k,z^k,\lambda^k)$ violation of the KKT system \eqref{PBBS_primal_inclusion}.
\end{proposition}
\begin{proof}
    The first--order optimality condition of \eqref{PBBS_x_update} grants
    \begin{equation}
    0\in\partial f(x^{k+1})+\rho A^*(Ax^{k+1}+\rho^{-1}\lambda^k).
    \label{PBBS_x_optimality}
    \end{equation}
    Since \eqref{PBBS_identity} holds for all $k$, the tuple $(x^{k+1},\mu^{k+1})$ is always such that $0\in\partial f(x^{k+1})+A^*\mu^{k+1}$.
    However, the tuple $(x^{k+1},\lambda^{k+1})$ does not enjoy the same property.
    Substituting \eqref{PBBS_dual_update} into \eqref{PBBS_x_optimality} yields 
    \begin{multline*}
    0\in\partial f(x^{k+1})+\rho A^{\top}(Ax^{k+1}+(\rho^{-1}\lambda^{k+1}-Ax^{k+1}-Bz^{k+1}-d))\\
    \iff\quad \rho A^{\top}(Ax^{k+1}+Bz^{k+1}+d)\in (\rho A^{\top}A+\partial f)x^{k+1}+A^{\top}\lambda^{k+1}.
    \end{multline*}
    We therefore view the quantity $\rho A^{\top}(Ax^k+Bz^k+d)$ as a \emph{dual residual}.
    Now considering \eqref{PBBS_z_update}, we have 
    $$
    0\in \partial g(z^{k+1})+\rho B^*(Ax^{k+1}+Bz^{k+1}+d+\rho^{-1}\lambda^k).
    $$
    It follows by substituting \eqref{PBBS_dual_update} that
    \begin{align*}
    0\in\partial g(z^{k+1})+\rho B^*(Ax^{k+1}+Bz^{k+1}+d+(\rho^{-1}\lambda^{k+1}-Ax^{k+1}-Bz^{k+1}-d))
    =\partial g(z^{k+1})+B^*\lambda^{k+1},
    \end{align*}
    which indicates that dual feasibility is always satisfied by $(z^{k+1},\lambda^{k+1})$. 
    In light of these results, we assess the feasibility of PBBS iterate $(x^k,z^k,\lambda^k)$ according to the quantities $\rho A^{\top}(Ax^k+Bz^k+d)$ and $Ax^k+Bz^k+d$. 
\end{proof}

In light of Proposition~\ref{prop:PBBS_residuals} and \eqref{PBBS_dual_update}, one sees that asymptotic regularity of the sequence $(\lambda^k)$ (i.e., $\|\lambda^{k+1}-\lambda^k\|\to0$) implies that the PBBS residuals converge to zero in norm and thus $(x^k,z^k,\lambda^k)\to(x,z,\lambda)$ such that $(x,z,\lambda)$ satisfies \eqref{PBBS_primal_inclusion}. 

\begin{theorem}
    Let Assumption~\ref{PBBS_assumptions} hold.
    Suppose furthermore that $-A^{\top}\mu\in\mathrm{int}(\dom f^*)$ and $-B^{\top}\lambda\in\mathrm{int}(\dom g^*)$ and $\partial f^*(-A^*\mu)$ and $\partial g^*(-B^*\lambda)$ are singletons.
    The sequence $(x^k,z^k,\lambda^k)$ generated according to \eqref{PBBS} converges to $(x,z,\lambda)$ such that \eqref{PBBS_primal_inclusion} holds.
\end{theorem}
\begin{proof}
    By Lemma~\ref{lem:BBS_convergent}, we have that $\lambda^k\to\lambda$ and $\mu^k\to\mu$. Lemma~\ref{lem:PBBS_convergent} grants that $x^k\to x\in \partial f^*(-A^*\mu)$ and $z^k\to z\in\partial g^*(-B^*\lambda)$.
    By the relation $\partial h^*=(\partial h)^{-1}$ for $h\in\Gamma (\H)$, we see that 
    $$
    0\in\partial f(x)+A^*\mu\quad\mand\quad0\in\partial g(z)+B^*\lambda.
    $$
    Since $\lambda^k-\mu^{k+1}=-\rho Ax^{k+1}$ holds for all $k$ and $A$ is linear, we have that $\mu=\lambda+\rho Ax$ and hence $0\in\partial f(x)+\rho A^*Ax+A^*\lambda$. 
    
    By \eqref{PBBS_dual_update} we have that 
    $$
    \|Ax^k+Bz^k+d\|=\rho^{-1}\|\lambda^{k+1}-\lambda^k\|,
    $$
    with the right hand side converging to zero by Lemma~\ref{lem:BBS_bounded}.
\end{proof}

\subsection{Numerical experiments}\label{sec:experiment_results}
We experimentally benchmark PBBS on both the classical and an $\ell_0$--regularized variant of the LASSO problem for linear predictive modeling.
We compare PBBS to the fast iterative shrinkage--thresholding algorithm (FISTA), which is a staple LASSO approach \cite{beck2009fista}.
We also compare PBBS and FISTA to a relaxation approach known as \emph{sparse relaxed regularized regression} (SR3).
SR3 is considered state--of--the--art in terms of numerical performance and efficiency for obtaining approximate solutions \cite{zheng2019sr3}.
We use performance profiles to capture the comparison (see \cite{more2002performanceprofiles} for more information). 

We compare the algorithms with eight regression instances, with the design and response data taken from the Faraway and scikit--learn datasets \cite{faraway2005data,pedregosa2011scikitlearn}.
The data is standardized by centering each feature to have zero mean and scaling to unit variance.
Each dataset is randomly sampled to form model training sets $U$ and $w$, where each algorithm sees the same training set in each instance.
The remaining samples $U_{\mathrm{test}}$ and $w_{\mathrm{test}}$ are used to validate each model's accuracy.
We report the mean squared error (MSE) $\frac1m\|U_{\mathrm{test}}\bar{x}-w_{\mathrm{test}}\|^2$ where $\bar{x}$ denotes the solution recovered by a method and $m$ is the number of features.

\subsubsection{LASSO}\label{sec:lasso}
The LASSO problem \cite{TibshiraniLASSO} is given by
\begin{equation}
\minimize_{x\in\R^n}\quad \frac12\|Ux-w\|^2+\lambda\|x\|_1
\label{lasso}
\end{equation}
for some regularization parameter $\lambda>0$, which we set to $0.1$ throughout.
Here $U\in\R^{m\times n}$ is a feature matrix with response vector $w\in\R^m$.
We relax and recast \eqref{lasso} for PBBS as
\begin{equation}
\minimize_{(x,z)\in\R^n\times\R^n}\quad \frac12\|Ux-w\|^2+\lambda\|z\|_1+\frac{\rho}2\|x\|^2\quad\subjectto\quad x-z=0.
\label{PBBS_lasso}
\end{equation}
In contrast, the SR3 approach relaxes \eqref{lasso} according to 
\begin{equation}
\minimize_{(x,z)\in\R^n\times\R^n}\quad \frac12\|Ux-w\|^2+\lambda\|z\|_1+\frac{\kappa}2\|x-z\|^2,
\label{SR3_lasso}
\end{equation}
where $\kappa>0$ controls the level of relaxation.
We note that the SR3 relaxation approaches \eqref{lasso} as $\kappa\to\infty$, while the PBBS relaxation approaches \eqref{lasso} as $\rho\to0$.
We empirically selected relaxation parameters $\rho$ and $\kappa$ for our experiments in order to have them recover solutions of comparable objective value to each other and FISTA.
We subsequently set $\rho=\frac1\kappa=0.1$ throughout.

Although expressions for PBBS iterate residuals are provided in Proposition~\ref{prop:PBBS_residuals}, we set the algorithm's stopping condition to be once $\|z^{k+1}-z^k\|\leq 10^{-6}$.
This choice is made in accordance with FISTA and SR3, which use iterate distance stopping conditions as standard.

We compare the MSE of the solutions obtained by each method in Figure~\ref{fig:lasso_mse}, where we observe that the quality of the models obtained by each method are almost indistinguishable in terms of MSE.
For this reason, which algorithm is better is entirely determined by which has a lesser run--time.

\begin{figure}[htb!]
\centering
\begin{tikzpicture}
\begin{axis}[
    ybar,
    bar width=7pt,
    width=0.9\textwidth,
    height=6cm,
    ylabel={MSE},
    xlabel={Data set},
    xtick=data,
    ymin=0,
    axis lines*=left,
    enlarge x limits=0.08,
    legend style={
        at={(0.815,0.85)},
        anchor=south,
        draw=none,
        legend columns=3,
        column sep=4pt
    },
    ymajorgrids=true,
    grid style=dashed,
]

\addplot[
    fill=cb2blue!80,
    draw=none
] table[
    x=dataset,
    y=PBBS,
    col sep=comma
] {l1mse.csv};

\addplot[
    fill=cb2red!80,
    draw=none
] table[
    x=dataset,
    y=SR3,
    col sep=comma
] {l1mse.csv};

\addplot[
    fill=cb2green!80,
    draw=none
] table[
    x=dataset,
    y=FISTA,
    col sep=comma
] {l1mse.csv};

\legend{PBBS, SR3, FISTA}

\end{axis}
\end{tikzpicture}
\caption{The MSE of solutions obtained by PBBS (left bar), SR3 (center bar) and FISTA (right bar) for eight instances of \eqref{lasso}.}
\label{fig:lasso_mse}
\end{figure}
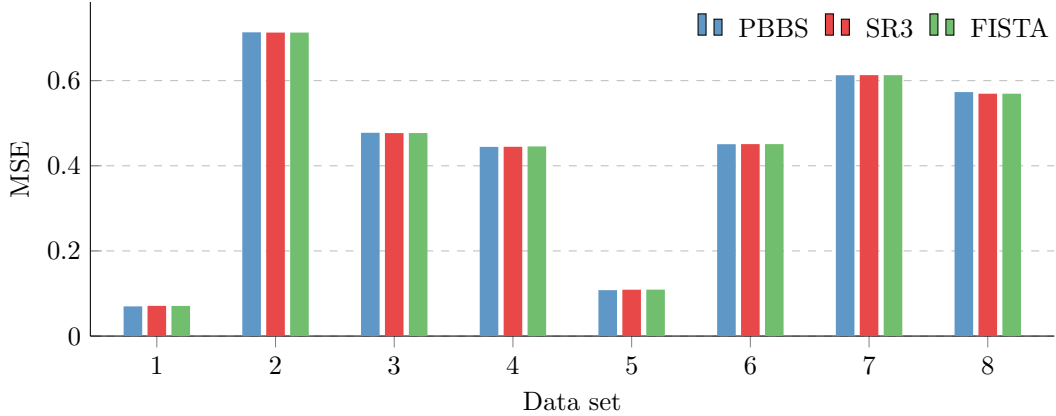

Performance profiles comparing the run time of PBBS against SR3 and FISTA are included in Figure~\ref{fig:lasso_PP}, where we observe PBBS to meet the stopping condition before both methods in every instance.
In Figure~\ref{fig:lasso_experiments} we plot consecutive iterate distances against time for a high--dimensional instance from \cite{uci2011ctscan} where $U\in\R^{48500\times 386}$.
We also plot the distance between iterates and higher--tolerance solutions obtained with stopping condition $\|z^k-z^{k-1}\|\leq10^{-8}$. We note that PBBS terminates after a single iteration in the instance of Figure~\ref{fig:lasso_experiments}.

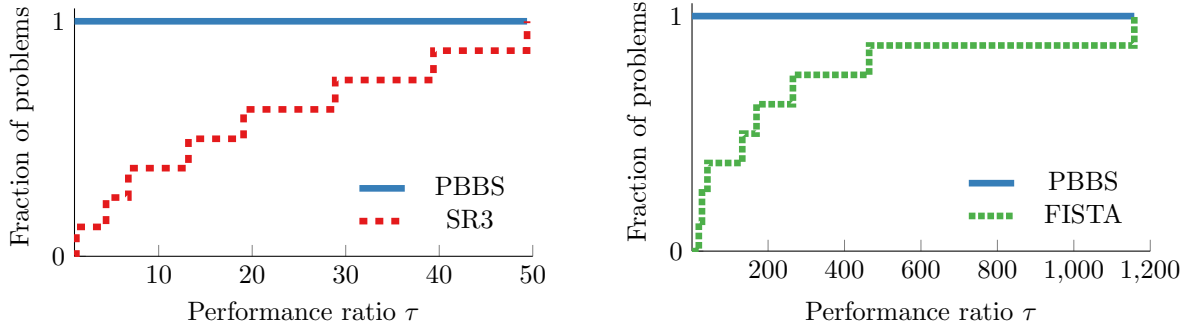
\begin{figure}[htb!]
\centering
\begin{subfigure}{0.48\textwidth}
\centering
\begin{tikzpicture}
\begin{axis}[
    width=\textwidth,
    height=5cm,
    axis lines*=left,
    xlabel={Performance ratio $\tau$},
    ylabel={Fraction of problems},
    xmin=1,
    xmax=50,
    ymin=0,
    ymax=1.1,
    ytick={0,1},
    y axis line style={
    shorten >=4%
    },
    ymajorgrids=false,
    grid style={gray!20},
    legend style={
        at={(0.97,0.05)},
        anchor=south east,
        draw=none,
        fill=none,
        legend columns=1,
        column sep=8pt
    },
]

\addplot[
    const plot,
    line width = 2.5pt,
    cb2blue
]
table[
    x=tau,
    y=PBBS,
    col sep=comma
] {l1PBBSvsSR3_PP.csv};

\addlegendentry{PBBS}

\addplot[
    const plot,
    line width=2.5pt,
    dashed,
    cb2red
]
table[
    x=tau,
    y=SR3,
    col sep=comma
] {l1PBBSvsSR3_PP.csv};

\addlegendentry{SR3}
\end{axis}
\end{tikzpicture}
\end{subfigure}
\hfill
\begin{subfigure}{0.48\textwidth}
\centering
\begin{tikzpicture}
\begin{axis}[
    width=\textwidth,
    height=5cm,
    axis lines*=left,
    xlabel={Performance ratio $\tau$},
    ylabel={Fraction of problems},
    xmin=1,
    xmax=1200,
    ymin=0,
    ymax=1.1,
    ytick={0,1},
    y axis line style={
    shorten >=4%
    },
    ymajorgrids=false,
    grid style={gray!20},
    legend style={
        at={(0.97,0.05)},
        anchor=south east,
        draw=none,
        fill=none,
        legend columns=1,
        column sep=8pt
    },
]

\addplot[
    const plot,
    line width = 2.5pt,
    cb2blue
]
table[
    x=tau,
    y=PBBS,
    col sep=comma
] {l1PBBSvsFISTA_PP.csv};

\addlegendentry{PBBS}

\addplot[
    const plot,
    line width=2.5pt,
    densely dotted,
    cb2green
]
table[
    x=tau,
    y=FISTA,
    col sep=comma
] {l1PBBSvsFISTA_PP.csv};

\addlegendentry{FISTA}
\end{axis}
\end{tikzpicture}
\end{subfigure}
\caption{Performance profiles comparing PBBS to SR3 and FISTA in terms of run time.}
\label{fig:lasso_PP}
\end{figure}

\begin{figure}[htb!]
\centering
\begin{subfigure}{0.48\textwidth}
\centering
\begin{tikzpicture}
\begin{axis}[
    width=\textwidth,
    height=5cm,
    xmode=log,
    ymode=log,
    xmin=5e-4,
    xmax=15, 
    ylabel={$\|z^k-z^{k-1}\|$},
    xlabel={Runtime (seconds)},
    scaled x ticks=false,
    xticklabel style={
    /pgf/number format/fixed,
    /pgf/number format/precision=3
},
    xmajorgrids=false,
    ymajorgrids=false,
    legend to name=combinedlegend,
    legend columns = 3,
    legend style={font=\scriptsize,draw=none},
]

\addplot[
    line width = 2.5pt,
    mark={},
    mark size=2pt,
    densely dotted,
    cb2green
] table[
    col sep = comma,
    x=x,
    y=y
] {l1FISTA_cons_dist.csv};
\addlegendentry{FISTA}

\addplot[
    line width = 2.5pt,
    mark={},
    mark size=2pt,
    cb2blue
] table[
    col sep=comma,
    x=x,
    y=y
] {l1PBBS_cons_dist.csv};
\addlegendentry{PBBS}

\addplot[
    line width = 2.5pt,
    mark={},
    mark size=2pt,
    dashed,
    cb2red
] table[
    col sep = comma,
    x=x,
    y=y
] {l1SR3_cons_dist.csv};
\addlegendentry{SR3}
\end{axis}
\end{tikzpicture}
\end{subfigure}
\hfill
\begin{subfigure}{0.48\textwidth}
\centering
\begin{tikzpicture}
\begin{axis}[
    width=\textwidth,
    height=5cm,
    xmode=log,
    ymode=log,
    xmin=5e-4,
    xmax=15,
    ylabel={$\|x^k-\bar{x}\|$},
    xlabel={Runtime (seconds)},
    scaled x ticks=false,
    xticklabel style={
    /pgf/number format/fixed,
    /pgf/number format/precision=3
},
    xmajorgrids=false,
    ymajorgrids=false,
    legend to name=combinedlegend2,
    legend columns = 3,
    legend style={font=\scriptsize,draw=none},
]

\addplot[
    line width = 2.5pt,
    mark={},
    mark size=2pt,
    densely dotted,
    cb2green
] table[
    col sep = comma,
    x=x,
    y=y
] {l1FISTA_sol_dist.csv};
\addlegendentry{FISTA}

\addplot[
    line width = 2.5pt,
    mark={},
    mark size=2pt,
    cb2blue
] table[
    col sep=comma,
    x=x,
    y=y
] {l1PBBS_sol_dist.csv};
\addlegendentry{PBBS}

\addplot[
    line width = 2.5pt,
    mark={},
    mark size=2pt,
    dashed,
    cb2red
] table[
    col sep = comma,
    x=x,
    y=y
] {l1SR3_sol_dist.csv};
\addlegendentry{SR3}
\end{axis}
\end{tikzpicture}
\end{subfigure}
\caption{Logarithmic plots comparing the convergence of PBBS, SR3 and FISTA for \eqref{lasso} with data taken from \cite{uci2011ctscan}.}
\label{fig:lasso_experiments}
\end{figure}
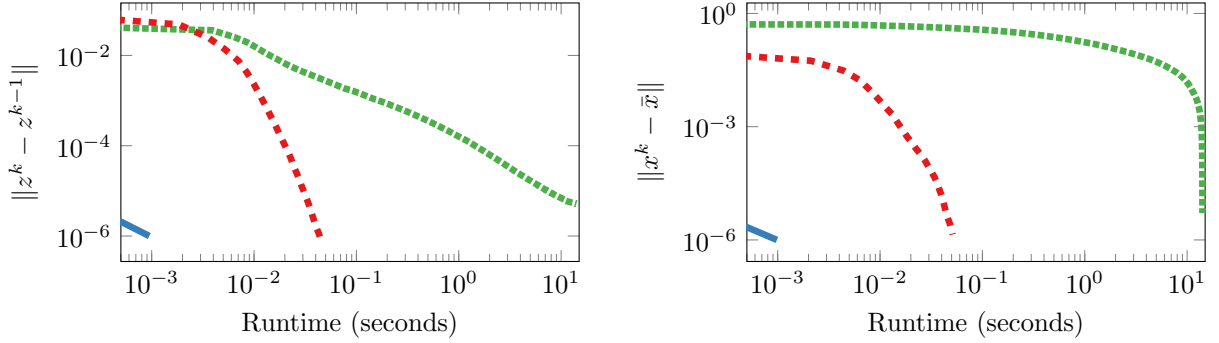

\subsubsection{Cardinality-regularized LASSO}\label{sec:l0_lasso} 
We also apply PBBS, SR3 and FISTA to the $\ell_0$--regularized regression problem
\begin{equation}
\minimize_{x\in\R^n}\quad \frac12\|Ux-w\|^2+\lambda\|x\|_0,
\label{l0_lasso}
\end{equation}
where $\|x\|_0=\sum_{i=1}^n|x_i|^0$ with the convention that $|x_i|^0=0$ if $x_i=0$ and $|x_i|^0=1$ otherwise.
We apply PBBS and SR3 to relaxations of \eqref{l0_lasso} akin to the relaxations in section~\ref{sec:lasso}.
In Figure~\ref{fig:l0_lasso_mse} we show the MSE of high tolerance solutions obtained by each method. We bring the reader's attention to FISTA for data set 4, where it appears to have significantly overfit the data.

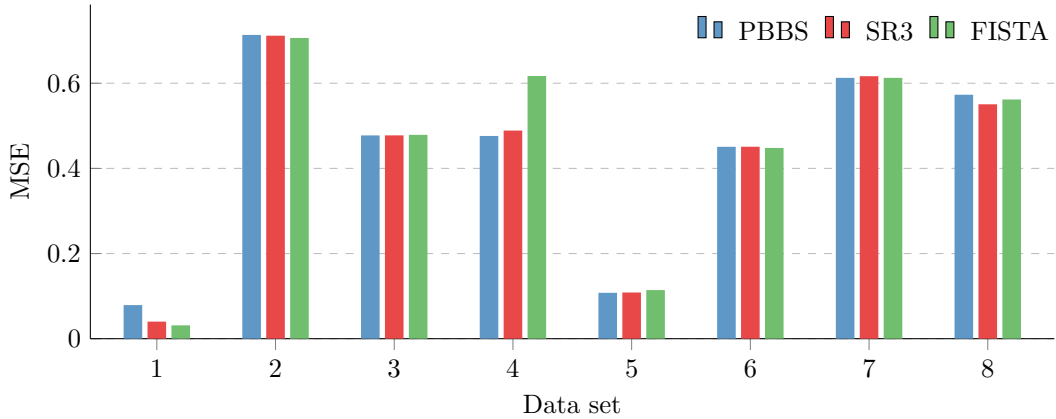
\begin{figure}[htb!]
\centering
\begin{tikzpicture}
\begin{axis}[
    ybar,
    bar width=7pt,
    width=0.9\textwidth,
    height=6cm,
    ylabel={MSE},
    xlabel={Data set},
    xtick=data,
    ymin=0,
    axis lines*=left,
    enlarge x limits=0.08,
    legend style={
        at={(0.815,0.85)},
        anchor=south,
        draw=none,
        legend columns=3,
        column sep=4pt
    },
    ymajorgrids=true,
    grid style=dashed,
]

\addplot[
    fill=cb2blue!80,
    draw=none,
] table[
    x=dataset,
    y=PBBS,
    col sep=comma
] {l0mse.csv};

\addplot[
    fill=cb2red!80,
    draw=none
] table[
    x=dataset,
    y=SR3,
    col sep=comma
] {l0mse.csv};

\addplot[
    fill=cb2green!80,
    draw=none
] table[
    x=dataset,
    y=FISTA,
    col sep=comma
] {l0mse.csv};

\legend{PBBS, SR3, FISTA}

\end{axis}
\end{tikzpicture}
\caption{The MSE of solutions obtained by PBBS (left bar), SR3 (center bar) and FISTA (right bar) for eight instances of \eqref{l0_lasso}.}
\label{fig:l0_lasso_mse}
\end{figure}

Although analyzing PBBS convergence for nonconvex problems is beyond the scope of this work, we empirically observe the method to perform well for \eqref{l0_lasso}. We also note that since problem \eqref{l0_lasso} is nonconvex, our work shows that the equivalence of PBBS and BBS still holds. 

We include performance profiles comparing PBBS with SR3 and FISTA for \eqref{l0_lasso} in Figure~\ref{fig:l0_PP}. We again compared the methods for the high-dimensional instance from \cite{uci2011ctscan} where $U \in \mathbb{R}^{48500 \times 386}$. The consecutive iterate differences and iterate distance to a solution for each method are given in Figure~\ref{fig:L0_lasso_experiments}. In this case, both PBBS and SR3 solved the problem in one step, obtaining solutions better than that obtained by FISTA for the test set. 

\begin{figure}[htb!]
\centering
\begin{subfigure}{0.48\textwidth}
\centering
\begin{tikzpicture}
\begin{axis}[
    width=\textwidth,
    height=5cm,
    axis lines*=left,
    xlabel={Performance ratio $\tau$},
    ylabel={Fraction of problems},
    xmin=1,
    xmax=45,
    ymin=0,
    ymax=1.1,
    ytick={0,1},
    y axis line style={
    shorten >=8%
    },
    ymajorgrids=false,
    grid style={gray!20},
    legend style={
        at={(0.97,0.05)},
        anchor=south east,
        draw=none,
        fill=none,
        legend columns=1,
        column sep=8pt
    },
]

\addplot[
    const plot,
    line width = 2.5pt,
    cb2blue
]
table[
    x=tau,
    y=PBBS,
    col sep=comma
] {l0PBBSvsSR3_PP.csv};

\addlegendentry{PBBS}

\addplot[
    const plot,
    line width=2.5pt,
    dashed,
    cb2red
]
table[
    x=tau,
    y=SR3,
    col sep=comma
] {l0PBBSvsSR3_PP.csv};

\addlegendentry{SR3}
\end{axis}
\end{tikzpicture}
\end{subfigure}
\hfill
\begin{subfigure}{0.48\textwidth}
\centering
\begin{tikzpicture}
\begin{axis}[
    width=\textwidth,
    height=5cm,
    axis lines*=left,
    xlabel={Performance ratio $\tau$},
    ylabel={Fraction of problems},
    xmin=1,
    xmax=350,
    ymin=0,
    ymax=1.1,
    ytick={0,1},
    y axis line style={
    shorten >=8%
    },
    ymajorgrids=false,
    grid style={gray!20},
    legend style={
        at={(0.97,0.05)},
        anchor=south east,
        draw=none,
        fill=none,
        legend columns=1,
        column sep=8pt
    },
]

\addplot[
    const plot,
    line width = 2.5pt,
    cb2blue
]
table[
    x=tau,
    y=PBBS,
    col sep=comma
] {l0PBBSvsFISTA_PP.csv};

\addlegendentry{PBBS}

\addplot[
    const plot,
    line width=2.5pt,
    densely dotted,
    cb2green
]
table[
    x=tau,
    y=FISTA,
    col sep=comma
] {l0PBBSvsFISTA_PP.csv};

\addlegendentry{FISTA}
\end{axis}
\end{tikzpicture}
\end{subfigure}
\caption{Performance profiles comparing PBBS to SR3 and FISTA in terms of run time for \eqref{l0_lasso}.}
\label{fig:l0_PP}
\end{figure}
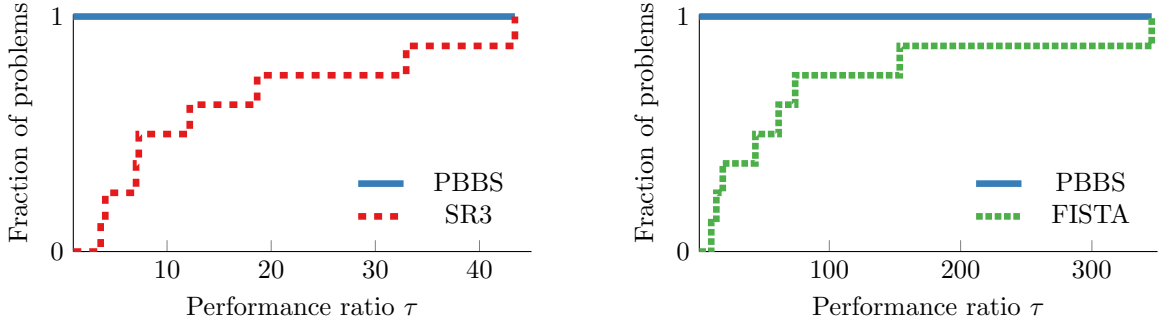

\begin{figure}[htb!]
    \centering
    \begin{subfigure}{0.48\textwidth}
        \centering
        \begin{tikzpicture}
        \begin{axis}[
            width=\textwidth,
            height=5cm,
            xmode=log,
            ymode=log,
            xmin=5e-4,
            xmax=15,
            ylabel={$\|z^k-z^{k-1}\|$},
            xlabel={Runtime (seconds)},
            scaled x ticks=false,
            xticklabel style={
            /pgf/number format/fixed,
            /pgf/number format/precision=3
        },
            xmajorgrids=false,
            ymajorgrids=false,
            legend to name=combinedlegend3,
            legend columns = 3,
            legend style={font=\scriptsize,draw=none},
        ]
        
        \addplot[
            line width = 2.5pt,
            mark={},
            mark size=2pt,
            densely dotted,
            cb2green
        ] table[
            col sep = comma,
            x=x,
            y=y
        ] {l0FISTA_cons_dist.csv};
        \addlegendentry{FISTA}
        
        \addplot[
            line width = 2.5pt,
            mark={},
            mark size=2pt,
            cb2blue
        ] table[
            col sep=comma,
            x=x,
            y=y
        ] {l0PBBS_cons_dist.csv};
        \addlegendentry{PBBS}
        
        \addplot[
            line width = 2.5pt,
            mark={},
            mark size=2pt,
            dashed,
            cb2red
        ] table[
            col sep = comma,
            x=x,
            y=y
        ] {l0SR3_cons_dist.csv};
        \addlegendentry{SR3}
        \end{axis}
        \end{tikzpicture}
    \end{subfigure}
    \hfill
    \begin{subfigure}{0.48\textwidth}
        \centering
        \begin{tikzpicture}
        \begin{axis}[
            width=\textwidth,
            height=5cm,
            xmode=log,
            ymode=log,
            xmin=5e-4,
            xmax=15,
            ylabel={$\|x^k-x^*\|$},
            xlabel={Runtime (seconds)},
            scaled x ticks=false,
            xticklabel style={
            /pgf/number format/fixed,
            /pgf/number format/precision=3
        },
            xmajorgrids=false,
            ymajorgrids=false,
            legend to name=combinedlegend4,
            legend columns = 3,
            legend style={font=\scriptsize,draw=none},
        ]
        
        \addplot[
            line width = 2.5pt,
            mark={},
            mark size=2pt,
            densely dotted,
            cb2green
        ] table[
            col sep = comma,
            x=x,
            y=y
        ] {l0FISTA_sol_dist.csv};
        \addlegendentry{FISTA}
        
        \addplot[
            line width = 2.5pt,
            mark={},
            mark size=2pt,
            cb2blue
        ] table[
            col sep=comma,
            x=x,
            y=y
        ] {l0PBBS_sol_dist.csv};
        \addlegendentry{PBBS}
        
        \addplot[
            line width = 2.5pt,
            mark={},
            mark size=2pt,
            dashed,
            cb2red
        ] table[
            col sep = comma,
            x=x,
            y=y
        ] {l0SR3_sol_dist.csv};
        \addlegendentry{SR3}
        \end{axis}
        \end{tikzpicture}
    \end{subfigure}
    \caption{Logarithmic plots comparing the convergence of PBBS, SR3 and FISTA for \eqref{l0_lasso} with data taken from \cite{uci2011ctscan}.}
    \label{fig:L0_lasso_experiments}
\end{figure}
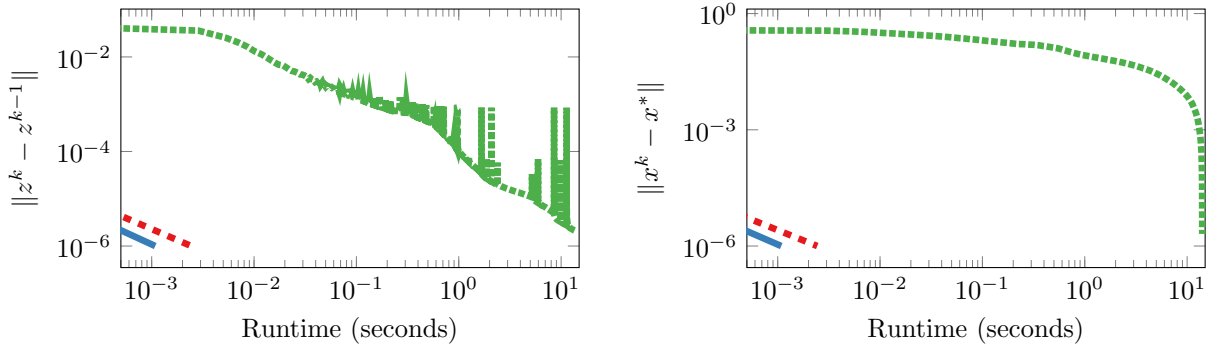

\section{Conclusion}

This work establishes a complete algorithmic duality framework to derive and explore the relationships between resolvent splitting methods and alternating minimization algorithms.
 The key is the Complete Moreau's Identity (Theorem~\ref{thm:Moreaus_identity}), which explicitly describes the relationships between primal and dual governing variables.
We demonstrate the framework by establishing a general duality between FBS and AMA and between DR and ADMM, extending the well known and fundamental relationship between these methods in the convex setting.
These results are key to, for instance, implementing principled primal--dual schemes and duality--based extrapolation schemes in nonconvex settings.

We also use our duality framework to develop and analyze a promising alternating minimization scheme based upon Yosida regularization, backward--backward splitting, and AMA. This approach, which we call PBBS, is naturally amenable to problems in statistics and machine learning. Our PBBS convergence analysis illustrates how Gabay duality and the Gabay identites provide a principled approach for studying alternating minimization algorithms. We use the Gabay framework to establish convergence guarantees for PBBS in the convex setting, identifying a case where AMA may converge without strong convexity assumptions. 

We validate the practical relevance of PBBS with numerical experiments on $\ell_1$ and $\ell_0$ regularized instances of LASSO for linear predictive modeling. We benchmark PBBS with the well--established FISTA and SR3 algorithms. Our numerical results demonstrate that PBBS consistently achieves solution quality comparable to that of FISTA and SR3 in less computational time.

\section*{Acknowledgements}

The MATRIX Institute workshop \emph{Splitting Algorithms - Advances, Challenges, and Opportunities} (February 2025) was instrumental to the success of this work.
Participation of all authors was made possible by generous financial support from MATRIX Institute, Simons Foundation, and Australian Mathematical Sciences Institute, as well as Australian Mathematical Society and its special interest group Mathematics of Computation and Optimisation, and an in-kind contribution from Drs.\ Minh N. Dao and Lien Nguyen.

\section*{Declarations}

\subsection*{Statement about AI use}
The authors declare that no artificial intelligence tools or large language models were used in the research, writing, or preparation of this manuscript.

\bibliographystyle{habbrv}
\bibliography{bibliography}

\end{document}